\documentclass[10pt,a4paper]{article}
\usepackage[a4paper,margin=2.3cm,includeheadfoot]{geometry}
\usepackage{amsmath,amssymb,amsthm,mathtools}
\usepackage[T1]{fontenc}
\usepackage{lmodern}
\usepackage{enumitem}
\usepackage{float}
\usepackage{tikz}
\usetikzlibrary{arrows.meta,positioning,fit,backgrounds,calc}
\usepackage{microtype}
\usepackage[hidelinks]{hyperref}

\hypersetup{pdftitle={Sequential-Innovation Reducibility and the Innovation Spectrum},pdfauthor={Amir Leshem},pdfsubject={Computability, reducibility, and algorithmic randomness}}

\title{Sequential-Innovation Reducibility and the Innovation Spectrum}
\author{Amir Leshem\\
\small Faculty of Engineering, Bar-Ilan University, Ramat Gan 52900, Israel\\
\small \texttt{amir.leshem@biu.ac.il}\\
\small ORCID: 0000-0002-2265-7463}
\date{}

\newtheorem{theorem}{Theorem}[section]
\newtheorem{lemma}[theorem]{Lemma}
\newtheorem{proposition}[theorem]{Proposition}
\newtheorem{corollary}[theorem]{Corollary}

\newtheorem{problem}[theorem]{Problem}
\theoremstyle{definition}
\newtheorem{definition}[theorem]{Definition}
\newtheorem{remark}[theorem]{Remark}

\newcommand{\lesir}{\leq_{\mathrm{sir}}}
\newcommand{\nlesir}{\nleq_{\mathrm{sir}}}

\newcommand{\supp}{\operatorname{supp}}
\newcommand{\Specinn}{\operatorname{Spec}_{\mathrm{inn}}}
\newcommand{\TTvis}{\operatorname{TTvis}}
\newcommand{\WTTvis}{\operatorname{WTTvis}}
\newcommand{\zeroSeq}{\overline{0}}
\newcommand{\edim}{\dim_{\mathrm{eff}}}

\newcommand{\MartinLof}{\mbox{Martin-L\"of}}
\newcommand{\MLR}{\operatorname{MLR}}
\newcommand{\equivsir}{\equiv_{\mathrm{sir}}}

\begin{document}
\maketitle
\begin{abstract}
 Sequential prediction naturally induces an innovation sequence consisting of the prediction errors produced by a causal predictor. We use the collection of all such innovation sequences to define the \emph{innovation spectrum} of an individual binary sequence and, from it, a new reducibility based on sequential information extraction. We show that this reducibility refines truth-table reducibility while exhibiting a fundamentally different geometry. The degree structure decomposes into two canonical regions: a truth-table spine, whose induced order is isomorphic to the truth-table degrees, and a complementary reservoir-immune region, consisting of sequences from which no infinite computable predictable reservoir can be extracted. We establish bridge constructions connecting the two regions, prove that reservoir immunity is preserved under sequential innovation, and show that the Martin--L\"of-random degrees form a proper downward-closed substructure inside the reservoir-immune region. These results reveal a new geometric organization of individual sequences based on causal predictability rather than classical oracle computation.
\end{abstract}



\noindent\textbf{Funding.} This research was partially supported by Israel Science Foundation grant ISF 2197/22.

\section{Introduction}

Classical oracle computation permits an algorithm to query a completed oracle
at positions chosen during the computation.  A causal sequential computation
has less freedom.  At stage $n$ it has seen only the prefix
$X\upharpoonright n$ and must commit before the next source bit $X(n)$ is
revealed.  We record the newly disclosed information through the innovation
\[
        I_P^X(n)=X(n)\oplus P(X\upharpoonright n),
\]
where $P:2^{<\omega}\to\{0,1\}$ is total computable.  If
$S=\{s_0<s_1<\cdots\}$ is an infinite computable schedule, the corresponding
trace is
\[
        J_{P,S}^X(k)=I_P^X(s_k).
\]
The schedule is fixed independently of the source values, and the prediction at
$s_k$ is made before $X(s_k)$ becomes available.

The resulting reducibility is related to several established restrictions on
oracle computation, but it is not identical to any of them.  In
truth-table autoreducibility, the current oracle bit is withheld, while the
machine may query other coordinates, including future coordinates
\cite{EbertMerkleVollmer2003}.  In computable Lipschitz and identity-bounded
Turing reducibility, the use is constrained by a rate such as $n+c$ or $n$,
and the induced degree structures already exhibit nontrivial density and
splitting phenomena
\cite{DowneyHirschfeldtLaForte2004,BarmpaliasLewis2006,Day2010}.  By contrast,
our use function may grow arbitrarily, but it must increase with the output
coordinate and the last queried bit must be genuinely fresh and balanced.
The model is also distinct from computable selection rules and
Kolmogorov--Loveland stochasticity: those notions study which source
coordinates are selected, possibly adaptively or nonmonotonically, whereas
our schedule is source-independent and every selected source bit participates
in a prescribed causal transformation
\cite{Schnorr1971,MerkleMillerNiesReimannStephan2006,Nies2009}.

The same causal viewpoint has a quantitative counterpart in
individual-sequence prediction. In companion work
\cite{leshem2026PRSuperpredictor}, a computable probabilistic
superpredictor achieves sublinear regret relative to every rational-valued
primitive-recursive forecaster on every individual sequence, and is Bayes
optimal on every Martin--L\"of-random realization of a computable stationary
ergodic binary source. Together, the two papers identify a causal
computational interface that is strictly weaker than unrestricted oracle
computation while still supporting universal prediction guarantees.
We define $Y\leq_{\mathrm{sir}}X$ when $Y$ is a sequential innovation trace
of $X$.  The first main result is an exact normal form: these are precisely
the truth-table computations with strictly increasing uses whose truth tables
flip whenever only the final queried bit is changed.  This immediately gives,
$X\lesir Y\Longrightarrow X\leq_{tt}Y
\Longrightarrow X\leq_{wtt}Y
\Longrightarrow X\leq_TY.$
The direction reverses conceptually after quotienting into degrees: the
associated equivalence relations become successively coarser, and the
identity on reals induces natural surjections
$
\mathcal D_{\mathrm{sir}}\twoheadrightarrow\mathcal D_{tt}
\twoheadrightarrow\mathcal D_{wtt}\twoheadrightarrow\mathcal D_T.
$
Thus SIR is a proper subreducibility of truth-table reducibility, while its
degree partition is finer: a single truth-table or Turing degree may split
into several SIR degrees. For comparison, Greenberg, Nies, and Turetsky \cite{GreenbergNiesTuretsky2026SJT} study
SJT reducibility, a weak reducibility arising from strong jump traceability
and partial relativization; its motivation and methods are distinct from the
causal framework developed here.
We briefly review the new results in this paper.

For a fixed presentation $X$, the innovation spectrum
\[
        \Specinn(X)=\{\deg_T(Y):Y\leq_{\mathrm{sir}}X\}
\]
records the Turing degrees exposed by causal traces.  This spectrum is an
invariant of the SIR degree, but it is not an invariant of the truth-table or
Turing degree.  The mechanism controlling its computable bottom is a
\emph{predictable reservoir}: an infinite computable set of future positions
whose bits are exactly predictable from their preceding prefixes.  A
presentation has a computable trace exactly when such a reservoir exists, and
then every truth-table reduction from that presentation can be realized as a
sequential trace.  Above the computable SIR degree, every target has a predictable reservoir,
so SIR and truth-table reducibility coincide.  Sparse computable padding
places a representative of every truth-table degree in this upper region.
Consequently this \emph{truth-table spine} is order-isomorphic to the full
truth-table degree structure, every truth-table degree has a unique greatest
SIR degree, and joins on the spine are induced by effective disjoint union.
Every c.e. presentation lies on the spine; in particular, c.e.-represented
SIR degrees are closed under joins. For a c.e. presentation $A$, the
presentations SIR-reducible to $A$ are exactly the presentations truth-table
reducible to $A$. Consequently, the truth-table degrees visible below $A$
form the principal truth-table ideal below $\deg_{tt}(A)$, and, after
restriction to the spine, this ideal is also precisely the SIR-degree
lower cone below $\deg_{\mathrm{sir}}(A)$.

The complementary region is exactly the class of reservoir-immune
presentations.  Indeed, a presentation lies above the computable SIR degree
if and only if it has a predictable reservoir, equivalently if and only if it
is not reservoir-immune.  Thus Cantor space splits exhaustively as
\[
2^\omega=\mathcal U_{\mathrm{sir}}\mathbin{\dot\cup}\mathrm{RI}.
\]
Both sides are dynamically natural: the truth-table spine is upward closed,
whereas reservoir immunity is downward closed under SIR.  Inside the RI
region, the fair-coin Martin--L\"of-random degrees form a proper
SIR-downward-closed subclass.  Every SIR trace of a random real is random,
but RI also contains nonrandom presentations in abundance: every basic
cylinder contains a perfect Cantor set of non-Martin--L\"of-random RI reals.

The second part develops the internal geometry of the random region.  For a
Martin--L\"of-random upper degree, every lower SIR degree is either a finite
causal shift in the same Turing degree or lies below a strict Turing-degree
drop.  Every interval of the second kind contains an order-isomorphic copy of
$(\mathbb Q,<)$ represented by random reals, whereas consecutive finite
shifts form cover pairs.  Moreover, the even and odd columns of one random
real are incomparable and have no SIR join.  At a broader level, every trace
random for a computable nondegenerate Bernoulli product has below it an
order-embedded copy of every computable countable partial order, represented
by reals random for the same product measure.  The bridge presentation
$\emptyset'\oplus R$, with $R$ $2$-random, combines the entire c.e. degree
structure with transverse random degrees and a lifted parallel family above
$\mathbf 0'$.

The main structure of the SIR degrees is summarized in the following theorem:
\begin{theorem}[Main structural theorem: exhaustive geometry and two bridges]
\label{thm:main-structure}
Let $\mathbf0_{\mathrm{sir}}$ denote the computable SIR degree, and put
\[
\begin{aligned}
\mathcal U_{\mathrm{sir}}
  &=\{\mathbf a:\mathbf0_{\mathrm{sir}}\leq_{\mathrm{sir}}\mathbf a\},\\
\mathcal I_{\mathrm{RI}}
  &=\{\deg_{\mathrm{sir}}(X):X\in\mathrm{RI}\},\\
\mathcal R_{\mathrm{sir}}
  &=\{\deg_{\mathrm{sir}}(R):R\in\MLR\}.
\end{aligned}
\]
Then:
\begin{enumerate}[label=(\roman*)]
\item As reducibility relations,
\[
\leq_{\mathrm{sir}}\subsetneq\leq_{tt}
\subsetneq\leq_{wtt}\subsetneq\leq_T.
\]
The induced degree quotients are related by natural surjections
\[
\mathcal D_{\mathrm{sir}}\twoheadrightarrow\mathcal D_{tt}
\twoheadrightarrow\mathcal D_{wtt}\twoheadrightarrow\mathcal D_T,
\]
and the first map is noninjective.
\item The upper region $\mathcal U_{\mathrm{sir}}$ is order-isomorphic to the
truth-table degrees and is an upper semilattice.  Every truth-table degree
contains a unique greatest SIR degree.  Every c.e. presentation belongs to
$\mathcal U_{\mathrm{sir}}$, and the c.e.-represented SIR degrees are closed
under joins.
\item The SIR degrees split exhaustively as
\[
\mathcal D_{\mathrm{sir}}
=\mathcal U_{\mathrm{sir}}\mathbin{\dot\cup}\mathcal I_{\mathrm{RI}}.
\]
The RI region $\mathcal I_{\mathrm{RI}}$ is downward closed.  The random
region $\mathcal R_{\mathrm{sir}}$ is a proper downward-closed subclass of
$\mathcal I_{\mathrm{RI}}$.  More strongly, every basic cylinder contains
a perfect null Cantor set contained in $\mathrm{RI}\setminus\MLR$.
\item If $B$ is Martin--L\"of random and $A\lesir B$, then either
$A\equiv_TB$ and $A\equivsir\sigma^kB$ for a unique $k<\omega$, or
$A<_TB$ and the open SIR interval between $A$ and $B$ contains an
order-isomorphic copy of $(\mathbb Q,<)$ represented by Martin--L\"of-random
reals.  Moreover, the even and odd columns of every Martin--L\"of-random real
are incomparable SIR degrees with no join.
\item There are two canonical bridge mechanisms.  For every
$X\in\mathrm{RI}$, its padded presentation $X^*$ lies on the spine,
$X\equiv_{tt}X^*$, and
\[
\deg_{\mathrm{sir}}(X)<\deg_{\mathrm{sir}}(X^*),
\]
where $X^*$ represents the unique greatest SIR degree in the truth-table
degree of $X$.  Separately, if $R$ is $2$-random, then
$\emptyset'\oplus R$ is a bridge presentation whose lower SIR cone contains
the c.e.-represented spine, transverse random degrees, and their lifted
families above $\mathbf0'$.
\item If $p\in(0,1)$ is computable and $X$ has a
$\mu_p$-Martin--L\"of-random SIR trace, then every computable countable
partial order embeds into the SIR degrees strictly below $X$, represented by
$\mu_p$-Martin--L\"of-random reals.
\end{enumerate}
\end{theorem}

The clauses are proved in Theorem~\ref{thm:sir-hierarchy},
Theorem~\ref{thm:tt-spine}, Corollary~\ref{cor:two-canonical-regions},
Theorem~\ref{thm:null-cantor-ri}, Theorems~\ref{thm:random-interval-dichotomy}
and~\ref{thm:random-no-join}, Theorem~\ref{thm:ce-random-bridge}, and
Theorem~\ref{thm:bernoulli-poset-universality}.  Together they exhibit an exhaustive partition into the canonical
semilattice spine and the downward-closed RI region, with the random degrees
forming a proper rigid subgeometry of RI.  The padding map $X\mapsto X^*$
provides the canonical bridge within each truth-table degree, while
$\emptyset'\oplus R$ provides a mixed c.e.--random bridge.  Below every
Bernoulli-random trace one additionally obtains universal computable order
structure.

Finally, a delayed batched version of the von Neumann extractor \cite{vonNeumann1951} shows that
every computable Bernoulli-random presentation has a fair-coin
Martin--L\"of-random SIR trace on a computable schedule of positive density.
Consequently, its lower SIR cone contains both a native
$\mu_p$-random universal suborder and an extracted fair-random universal
suborder.  More efficient block extractors, such as those based on Elias's
type-class method \cite{Elias1972}, can improve the density, but rate optimization is not
needed for the degree-theoretic conclusions developed here.

The paper is organized as follows.  Section~2 introduces the model and proves
the exact trace characterization.  Section~3 places SIR among classical
reducibilities.  Sections~4 and~5 establish realization, reservoir immunity, the truth-table
spine, and presentation-sensitivity results.  Section~6 proves Bernoulli-random
universality, the fair-random interval dichotomy, the no-join theorem, and the
c.e.--random bridge.  Section~7 proves the Bernoulli extraction theorem.

\section{Sequential innovation: model and definitions}
\subsection{Predictors, traces, and calibration}
We work in Cantor space $2^\omega$ with fair-coin measure $\lambda$.  For a
finite string $\sigma\in2^{<\omega}$, let
\[
        [\sigma]=\{X\in2^\omega:\sigma\preceq X\}.
\]
All relativized notions have their standard meanings.  We use $K$ for
prefix-free Kolmogorov complexity and $K^A$ for its relativization.  The
effective Hausdorff dimension and strong effective dimension of a real are
\[
        \edim(X)=\liminf_{n\to\infty}
        \frac{K(X\upharpoonright n)}{n},
        \qquad
        \operatorname{Dim}_{\mathrm{eff}}(X)=\limsup_{n\to\infty}
        \frac{K(X\upharpoonright n)}{n}.
\]

\begin{theorem}[Levin--Schnorr]
For every $X\in2^\omega$ and oracle $A$,
\[
X\text{ is $\MartinLof$ random relative to }A
\quad\Longleftrightarrow\quad
K^A(X\upharpoonright n)\geq n-O(1).
\]
\end{theorem}
Hence, the effective dimension of a $\MartinLof$ random real is 1.
\begin{definition}[Predictor and innovation]
\label{def:PI}
An $A$-computable predictor is a total $A$-computable function
$P:2^{<\omega}\to\{0,1\}$.  For $X\in2^\omega$, define
\begin{equation}
\label{def:EP}
E_P(X)=\{n:P(X\upharpoonright n)\neq X(n)\},
\end{equation}
its complement 
\begin{equation}
\label{def:CP}
C_P(X)=\{n:P(X\upharpoonright n) = X(n)\},
\end{equation}
and
\[
I_P^X(n)=X(n)\oplus P(X\upharpoonright n).
\]
Thus $E_P(X)=\supp(I_P^X)$.
\end{definition}

Totality is essential: a limit-computable approximation is not an online rule
unless its final value is available at prediction time.

\begin{theorem}[Finite innovation equals relative computability]
\label{thm:finite-innovation}
For all $A,X\in2^\omega$, the following are equivalent:
\begin{enumerate}[label=(\roman*)]
\item $X\leq_TA$;
\item some total $A$-computable predictor makes only finitely many errors on
$X$;
\item some total $A$-computable predictor has innovation sequence of finite
support on $X$.
\end{enumerate}
\end{theorem}

\begin{proof}
If $X\leq_TA$, use the predictor $P(\sigma)=X(|\sigma|)$.  Conversely, suppose
$P\leq_TA$ and all errors occur before $N$.  Hardwire
$X\upharpoonright N$.  Once $X\upharpoonright n$ has been reconstructed for
$n\geq N$, set $X(n)=P(X\upharpoonright n)$.  This computes $X$ from $A$.
The final equivalence follows from $E_P(X)=\supp(I_P^X)$.
\end{proof}

\begin{proposition}[Innovation homeomorphism]
\label{prop:innovation-homeomorphism}
For every computable predictor $P$, the map
\[
        \Phi_P(X)=I_P^X
\]
is a computable measure-preserving homeomorphism of $2^\omega$.  Indeed,
\[
        I_P^X\equivsir X,
\]
and hence $I_P^X\equiv_TX$.  Moreover,
\[
        K(I_P^X\upharpoonright n)=K(X\upharpoonright n)+O_P(1).
\]
\end{proposition}

\begin{proof}
The forward reduction $I_P^X\lesir X$ is the full-schedule channel with
predictor $P$.  Given a finite innovation prefix $\tau$, reconstruct the
corresponding source prefix $\widehat\tau$ recursively and put
\[
        Q(\tau)=P(\widehat\tau).
\]
Then
\[
        X(n)=I_P^X(n)\oplus Q(I_P^X\upharpoonright n),
\]
so $X\lesir I_P^X$ by the full schedule.  At every length $n$, the induced
map on $2^n$ is a computable bijection, so it preserves the uniform measure on
length-$n$ cylinders.  The complexity identity follows from the computable
inverse pair.
\end{proof}
A naive way to sparsify a sequence $X$ is to place its bits on an infinite computable set $S={s_0<s_1<\cdots}$, defining
$
Y(s_k)=X(k)
$
and assigning a fixed value $0$ to all remaining coordinates. This preserves the encoded sequence but introduces an artificial computable pattern on the unused coordinates and may therefore destroy the randomness properties that motivated the construction. To retain the sequential nature of the source, we study all computable prediction channels:
\begin{definition}[Prediction channel and trace]
A prediction channel is a pair $(P,S)$, where $P$ is a total computable
predictor and $S=\{s_0<s_1<\cdots\}$ is an infinite computable set. Its trace
on $X$ is
\[
J_{P,S}^X(k)=X(s_k)\oplus P(X\upharpoonright s_k).
\]
The trace deficiency is
\[
d_{P,S}(X,k)=k-K(J_{P,S}^X\upharpoonright k).
\]
\end{definition}

The coordinate projection onto an infinite computable schedule is a computable
measure-preserving map from $2^\omega$ to $2^\omega$.  Together with
Proposition~\ref{prop:innovation-homeomorphism}, this gives the following.

\begin{theorem}[Residual Levin--Schnorr calibration]
\label{thm:residual-ls}
For $X\in2^\omega$, the following are equivalent:
\begin{enumerate}[label=(\roman*)]
\item $X$ is $\MartinLof$ random;
\item every computable sequential innovation trace $J_{P,S}^X$ is $\MartinLof$ random;
\item for every computable prediction channel $(P,S)$,
\[
        K(J_{P,S}^X\upharpoonright k)\geq k-O_{P,S}(1);
\]
\item every trace deficiency $d_{P,S}(X,k)$ is bounded.
\end{enumerate}
The same equivalence holds relative to an arbitrary oracle $A$.
\end{theorem}

\begin{proof}
A computable channel is a computable measure-preserving map, so it conserves
$\MartinLof$ randomness.  The complexity equivalence follows from
Levin--Schnorr.  Conversely, the trivial channel $P\equiv0$, $S=\omega$ has
trace $X$.
\end{proof}

\begin{corollary}[Sequential procedures do not derandomize]
\label{cor:random-sir-lower-region}
Let $R$ be fair-coin Martin--L\"of random.  If $Y\lesir R$, then $Y$ is
fair-coin Martin--L\"of random.  Consequently,
\[
\mathcal R_{\mathrm{sir}}
=\{\deg_{\mathrm{sir}}(R):R\in\MLR\}
\]
is downward closed in the SIR degrees.  In particular, no computable or
otherwise non-Martin--L\"of-random real is SIR-reducible to $R$.
\end{corollary}

\begin{proof}
Write $Y=J_{P,S}^R$.  By
Proposition~\ref{prop:innovation-homeomorphism}, the map
$R\mapsto I_P^R$ is a computable measure-preserving homeomorphism of fair-coin
Cantor space.  Projection onto the infinite computable schedule $S$ is also a
computable measure-preserving map onto fair-coin Cantor space.  Randomness
conservation therefore gives $Y\in\MLR$.  Downward closure on degrees follows
from transitivity and from invariance of Martin--L\"of randomness under SIR
equivalence.
\end{proof}

Theorem~\ref{thm:residual-ls} is a calibration theorem: the universal
quantifier includes the identity channel.  The more structural issue is to
characterize the transformations represented by nontrivial channels.

\subsection{Sequential-innovation reducibility}
\begin{definition}[Sequential-innovation reducibility]
Let $X,Y\in2^\omega$. We write
\[
Y\leq_{\mathrm{sir}}X
\]
if there exist a total computable strictly increasing function
\[
s:\omega\to\omega
\]
and a uniformly computable family of truth tables
\[
\theta_k:2^{s(k)+1}\to\{0,1\}
\]
such that, for every $k$, every $\sigma\in2^{s(k)}$, and every
$b\in\{0,1\}$,
\[
\theta_k(\sigma b)=b\oplus\theta_k(\sigma0),
\]
and
\[
Y(k)=\theta_k(X\upharpoonright(s(k)+1)).
\]
\end{definition}
The final query is distinguished.  The output may depend arbitrarily on the
past prefix, but changing only the current scheduled bit must flip the output.

\begin{theorem}[Exact trace characterization]
\label{thm:exact-trace-characterization}
For reals $X,Y\in2^\omega$, the following are equivalent:
\begin{enumerate}[label=(\roman*)]
\item There are a computable strictly increasing function $s$ and a family
$(\theta_k)_{k\in\omega}$ of last-bit-balanced truth tables, uniformly
computable in $k$, such that
\[
Y(k)=\theta_k(X\upharpoonright(s(k)+1))
\]
for every $k$.
\item There is a computable prediction channel $(P,S)$ such that
\[
Y=J_{P,S}^X.
\]
\end{enumerate}
\end{theorem}
\begin{proof}
Suppose $Y=J_{P,S}^X$, where $S=\{s_0<s_1<\cdots\}$.  Define
\[
        \theta_k(\sigma b)=b\oplus P(\sigma)
        \qquad(|\sigma|=s_k).
\]
The family is uniformly computable and last-bit balanced, and
\[
\theta_k(X\upharpoonright(s_k+1))
=X(s_k)\oplus P(X\upharpoonright s_k)=Y(k).
\]

Conversely, suppose $Y\lesir X$ via $s$ and $(\theta_k)$.  For
$|\sigma|=s(k)$, let
\[
        P(\sigma)=\theta_k(\sigma0).
\]
At all other lengths set $P(\sigma)=0$.  Since $s$ is computable and strictly
increasing, $P$ is total computable.  Last-bit balance gives
\[
\theta_k(\sigma b)=b\oplus P(\sigma).
\]
Thus, for $S=\{s(k):k\in\omega\}$,
\[
Y(k)=X(s(k))\oplus P(X\upharpoonright s(k))=J_{P,S}^X(k).
\]
\end{proof}

\begin{proposition}[Sequential-innovation reducibility is a preorder]
\label{prop:sir-preorder}
The relation $\lesir$ is reflexive and transitive.
\end{proposition}

\begin{proof}
Reflexivity is witnessed by $s(k)=k$ and
$\theta_k(\sigma b)=b$, equivalently by the trivial channel.

For transitivity, suppose
\[
Y=J_{P,S}^X,
\qquad
Z=J_{Q,T}^Y,
\]
where $S=\{s_0<s_1<\cdots\}$ and $T=\{t_0<t_1<\cdots\}$.  Put
$u_j=s_{t_j}$.  Given a string $\sigma$ of length $u_j$, compute the finite
string $Y_\sigma\upharpoonright t_j$ by
\[
Y_\sigma(i)=\sigma(s_i)\oplus P(\sigma\upharpoonright s_i)
\qquad(i<t_j).
\]
Define, at length $u_j$,
\[
R(\sigma)=P(\sigma)\oplus Q(Y_\sigma\upharpoonright t_j),
\]
and define $R$ arbitrarily at all other lengths.  Then
\[
\begin{aligned}
J_{R,U}^X(j)
&=X(s_{t_j})\oplus P(X\upharpoonright s_{t_j})
   \oplus Q(Y\upharpoonright t_j)\\
&=Y(t_j)\oplus Q(Y\upharpoonright t_j)\\
&=Z(j).
\end{aligned}
\]
Hence $Z\lesir X$.
\end{proof}

This closure is a useful distinction from an ad hoc family of coordinate maps:
sequential innovation traces form a natural sequential reduction preorder.

\section{Position among classical reducibilities}
\begin{definition}[Innovation spectrum]
For $X\in2^\omega$, define
\[
        \Specinn(X)
        =\{\deg_T(J_{P,S}^X):(P,S)\text{ is a computable channel}\}.
\]
By Theorem~\ref{thm:exact-trace-characterization},
\[
        \Specinn(X)=\{\deg_T(Y):Y\lesir X\}.
\]
\end{definition}

\begin{proposition}[SIR-degree invariance of the spectrum]
\label{prop:spectrum-sir-invariant}
If $X\equivsir Y$, then
\[
        \Specinn(X)=\Specinn(Y).
\]
Thus the innovation spectrum is an invariant of the sequential-innovation
degree.
\end{proposition}

\begin{proof}
If $Z\lesir X$ and $X\lesir Y$, transitivity gives $Z\lesir Y$; the reverse
inclusion is symmetric.
\end{proof}
The innovation spectrum is not, however, an invariant of the truth-table degree, as will be shown below.

For a presentation $X$, define its truth-table-visible degree set by
\[
        \TTvis(X)=\{\deg_T(B):B\leq_{tt}X\}.
\]
We avoid the term ``cone,'' since this set need not be downward closed under
arbitrary Turing reducibility.

\begin{proposition}[Truth-table visibility bound]
\label{prop:tt-bound}
If $Y\lesir X$, then $Y\leq_{tt}X$.  Consequently,
\[
        \Specinn(X)\subseteq\TTvis(X).
\]
\end{proposition}

\begin{proof}
If $Y=J_{P,S}^X$ and $S=\{s_k\}$, then
\[
        Y(k)=X(s_k)\oplus P(X\upharpoonright s_k)
\]
is computable from $X\upharpoonright(s_k+1)$.  The use $s_k+1$ is computable.
\end{proof}

\begin{theorem}[Reducibility inclusion and degree refinement]
\label{thm:sir-hierarchy}
As binary relations on Cantor space,
\[
\leq_{\mathrm{sir}}\Rightarrow\leq_{tt}
\Rightarrow\leq_{wtt}\Rightarrow\leq_T.
\]
Accordingly,
\[
\equiv_{\mathrm{sir}}\Rightarrow\equiv_{tt}
\Rightarrow\equiv_{wtt}\Rightarrow\equiv_T.
\]
Writing $\mathcal D_\rho$ for the quotient degrees induced by a reducibility
$\rho$, the identity on reals induces natural surjections
\[
\mathcal D_{\mathrm{sir}}\twoheadrightarrow\mathcal D_{tt}
\twoheadrightarrow\mathcal D_{wtt}\twoheadrightarrow\mathcal D_T.
\]
The first map is noninjective: some truth-table degrees split into multiple
SIR degrees.
\end{theorem}

\begin{proof}
The first inclusion is Proposition~\ref{prop:tt-bound}, and its strictness is
proved in Corollary~\ref{cor:sir-proper}.  The remaining strict inclusions are
the classical separation of truth-table, weak truth-table, and Turing
reducibility; see \cite{Rogers1967,Odifreddi1989}.  The inclusions of equivalence relations follow immediately.  They make the
quotient maps well defined, and surjectivity follows because every degree is
represented by a real.  Noninjectivity of the first quotient map is witnessed
inside one truth-table degree by
Corollary~\ref{cor:tt-degree-splits-sir}; the remaining classical degree
collapses are standard.
\end{proof}
The distinction between truth-table and sequential-innovation reducibility is therefore not one of finite use, but of causality: at the selected coordinate, the output must depend on the current source bit in a balanced way determined by the preceding history. 
The exact trace characterization places sequential-innovation reducibility inside
classical bounded-use computability.  The comparison separates two genuinely
different sources of nonuniformity: sequential balance at the newly revealed bit and
partiality on counterfactual oracle prefixes.

For a presentation $X$, define
\[
\WTTvis(X)=\{\deg_T(B):B\leq_{wtt}X\}.
\]
Then
\[
\Specinn(X)\subseteq\TTvis(X)\subseteq\WTTvis(X).
\]
The first inclusion is Proposition~\ref{prop:tt-bound}.  The second is the
standard implication from truth-table to weak truth-table reducibility.
A useful normal form is the following.  A weak truth-table reduction may be
written as a computable use function $u(k)$ together with uniformly partial
computable maps
\[
 F_k:2^{u(k)}\rightharpoonup\{0,1\}
\]
such that $F_k(X\upharpoonright u(k))=B(k)$.  For truth-table reducibility each
$F_k$ must be total. For sequential-innovation reducibility, after increasing the use if necessary,
there are total computable maps $q_k$ and a strictly increasing computable
schedule $(s_k)_{k\in\omega}$ such that
\begin{equation}
B(k)=X(s_k)\oplus q_k(X\upharpoonright s_k).
\end{equation}
Thus sequential-innovation reducibility is precisely truth-table computation
subject to causality relative to a computable schedule: the value
$q_k(X\upharpoonright s_k)$ is determined before the current source bit
$X(s_k)$ is revealed.
Thus the final input bit is fresh and balanced: changing only $X(s_k)$ flips
the output.

\begin{remark}[Standard bounded-use totalization]
\label{rem:wtt-totalization}
The standard comparison of bounded reducibilities gives
\[
 B\leq_{wtt}X
 \quad\Longrightarrow\quad
 B\leq_{tt}X\oplus\emptyset'.
\]
Indeed, the computable use bound leaves only finitely many counterfactual
oracle-answer patterns on each input, and $\emptyset'$ decides which of the
corresponding finite computations halt; see
\cite{Rogers1967,Odifreddi1989}.  Consequently, if
$\emptyset'\leq_{tt}X$, then
\[
 \TTvis(X)=\WTTvis(X).
\]
\end{remark}

The totality requirement on the sequential rule is essential.  If the rule
were required to converge only along the actual prefixes of \(X\), the
counterfactual truth tables could remain partial and the resulting notion
would be closer to weak truth-table computation than to the uniform
sequential transformation studied here.  Strictness of the first inclusion
will follow from reservoir-immune presentations in
Section~\ref{sec:reservoir-visibility}.

\section{Elementary degree structure and realization}

The identity channel shows that \(\deg_T(X)\) is always the largest degree
represented in \(\Specinn(X)\).  At the opposite extreme, thin Turing lower
cones sharply restrict the spectrum.

\begin{proposition}[Minimal degrees]
If $X$ has minimal nonzero Turing degree $\mathbf a$, then
\[
        \Specinn(X)\subseteq\{\mathbf 0,\mathbf a\}.
\]
Since $\mathbf a\in\Specinn(X)$, its spectrum is either
$\{\mathbf a\}$ or $\{\mathbf0,\mathbf a\}$.
\end{proposition}

\subsection{Sparse and countable realization}
The fixed-presentation results above constrain traces of a given real.  We now
turn to the complementary construction problem: which patterns can be built
into a suitable presentation?

\begin{definition}[Sparse coding]
Let $S=\{s_0<s_1<\cdots\}$ be infinite and computable.  Define
\[
\operatorname{Sparse}_S(B)(s_k)=B(k),
\qquad
\operatorname{Sparse}_S(B)(n)=0\quad(n\notin S).
\]
\end{definition}

\begin{theorem}[Sparse innovation representation]
\label{thm:sparse-representation}
Let $P$ be a total computable predictor and let $S$ be an infinite computable
schedule.  For every $B\in2^\omega$ there is $X_B$ such that
\[
        I_P^{X_B}=\operatorname{Sparse}_S(B),
        \qquad
        J_{P,S}^{X_B}=B,
\]
and
\[
        X_B\equiv_TB.
\]
\end{theorem}

\begin{proof}
Construct $X_B$ recursively.  If $n\notin S$, set
\[
        X_B(n)=P(X_B\upharpoonright n).
\]
If $n=s_k$, set
\[
        X_B(s_k)=P(X_B\upharpoonright s_k)\oplus B(k).
\]
The displayed trace identities are immediate.  The construction computes
$X_B$ from $B$, while $B=J_{P,S}^{X_B}$ is computable from $X_B$.
\end{proof}

This representation separates ambient complexity from trace complexity.  If
$S$ has counting function $c_S(N)=|S\cap N|$, then
\[
K(\operatorname{Sparse}_S(B)\upharpoonright N)
\leq K(B\upharpoonright c_S(N))+K(N)+O(1).
\]
Thus a density-zero presentation may have arbitrarily small ambient complexity
rate while its normalized trace is $B$.

\begin{proposition}[Arbitrarily slow ambient complexity in a degree]
Let $B\in2^\omega$ and let $g:\omega\to\omega$ be computable,
nondecreasing, and unbounded.  There is $X\equiv_TB$ such that
\[
        K(X\upharpoonright N\mid N)\leq g(N)+O(1)
\]
for all sufficiently large $N$.
\end{proposition}

\begin{proof}
Choose an infinite computable schedule $S$ with
$c_S(N)\leq g(N)$ eventually and put
$X=\operatorname{Sparse}_S(B)$.  Given $N$, the prefix
$X\upharpoonright N$ is determined by $N$ and the first $c_S(N)$ bits of $B$.
\end{proof}

For simultaneous realization, the schedule family must have an effective owner
map.  Fix a computable pairing function $\langle i,k\rangle$ such that, for
each fixed $i$, the map $k\mapsto\langle i,k\rangle$ is strictly increasing.
Let
\[
        C_i=\{\langle i,k\rangle:k\in\omega\}.
\]
The columns $(C_i)$ partition $\omega$, and both the owner $i$ and rank $k$ of
each coordinate are computable.

\begin{theorem}[Countable decoded-column realization]
\label{thm:countable-realization}
Let $P$ be a total computable predictor and let $(B_i)_{i\in\omega}$ be any
sequence of reals.  There is a real $X$ such that
\[
        J_{P,C_i}^X=B_i
        \qquad\text{for every }i,
\]
and
\[
        X\equiv_T\bigoplus_{i\in\omega}B_i.
\]
Consequently,
\[
        \{\deg_T(B_i):i\in\omega\}\subseteq\Specinn(X).
\]
\end{theorem}

\begin{proof}
At coordinate $n=\langle i,k\rangle$, set recursively
\[
        X(n)=P(X\upharpoonright n)\oplus B_i(k).
\]
The inverse pairing map makes the construction effective from
$\bigoplus_iB_i$.  It gives $J_{P,C_i}^X=B_i$.  Conversely, all columns $B_i$
are uniformly recoverable from $X$, so their effective join is computable from
$X$.
\end{proof}

\begin{remark}
A uniformly computable family of pairwise disjoint computable sets need not have
a decidable union or a computable owner map.  Pairwise disjointness gives
uniqueness of an owner, not an effective method for finding it.  The canonical
columns above make the recursive construction total and uniform.
\end{remark}

\begin{corollary}[Countable complexity profiles]
For any sequence $(B_i)$, a single presentation can satisfy
\[
        K(J_{P,C_i}^X\upharpoonright k)
        =K(B_i\upharpoonright k)+O_i(1)
\]
for every $i$.  Thus computable, $K$-trivial, intermediate-dimension,
\MartinLof-random, and complete traces can coexist in one suitably constructed
nonrandom presentation.
\end{corollary}

\begin{corollary}[Recursive tree realization]
Let $T\subseteq\omega^{<\omega}$ be a computable tree and let
$(B_\tau)_{\tau\in T}$ be a family of reals with effective join.  Then some
real $X$ has
\[
        \{\deg_T(B_\tau):\tau\in T\}\subseteq\Specinn(X),
\]
and $X$ is Turing equivalent to the effective join of the family.
\end{corollary}

\begin{proof}
Enumerate the nodes of $T$ computably and apply
Theorem~\ref{thm:countable-realization} to the corresponding canonical columns.
\end{proof}

\section{Reservoir immunity, presentation sensitivity, and visibility}
\label{sec:reservoir-visibility}

The obstruction to representing an arbitrary truth-table computation is the
unavailable current bit.  A truth-table evaluator may use that bit in an
arbitrary way, whereas a sequential rule must commit before seeing it.  A
computably predictable set of future sites removes this obstruction.

\begin{definition}[Predictable reservoir]
An infinite computable set $C=\{c_0<c_1<\cdots\}$ is a \emph{predictable
reservoir} for $X$ if there is a total computable predictor $Q$ such that
\[
        Q(X\upharpoonright c_j)=X(c_j)
\]
for every $j$.
\end{definition}

\begin{lemma}[Two-sided infinitude]
\label{lem:two-sided-infinitude}
Let $X$ be noncomputable and let $P$ be a total computable predictor.  Then
both
$C_P(X), E_P(X)$ defined in Definition~\ref{def:PI}
are infinite.
\end{lemma}

\begin{proof}
If $E_P(X)$ were finite, then the full innovation sequence
\[
 I_P^X(n)=X(n)\oplus P(X\upharpoonright n)
\]
would have finite support.  Hard-coding that finite support and reconstructing
$X$ recursively from left to right would make $X$ computable.  If $C_P(X)$
were finite, the complementary predictor
$\overline P(\sigma)=1-P(\sigma)$ would make only finitely many errors, and
the same argument would apply.
\end{proof}

Recall that an infinite set $A\subseteq\omega$ is \emph{immune} if it has no
infinite computably enumerable subset, and is \emph{bi-immune} if both $A$
and its complement are immune~\cite{Rogers1967,Odifreddi1989}.  Every infinite
c.e. set contains an infinite computable subset.  Hence, for an infinite set
$A$, immunity is equivalent to the absence of an infinite computable subset
of $A$.  This equivalence concerns inclusion of subsets of $A$ and makes no
assertion about Turing-degree comparability between $A$ and the c.e. sets.

\begin{definition}[Reservoir immunity]
A real $X$ is \emph{reservoir-immune} if, for every total computable predictor
$P$, both $C_P(X)$ and $E_P(X)$ are immune.  Equivalently, every full
innovation sequence $I_P^X$ is bi-immune.
\end{definition}

\begin{proposition}[Computable traces and predictable reservoirs]
\label{prop:computable-trace-reservoir}
For every real $X$, the following are equivalent:
\begin{enumerate}
\item $X$ has a predictable reservoir;
\item $\zeroSeq\lesir X$;
\item $\mathbf 0\in\Specinn(X)$;
\item Some computable sequential innovation trace is obtainable from $X$;
\item $X$ is not reservoir-immune.
\end{enumerate}
\end{proposition}

\begin{proof}
A predictable reservoir immediately yields the constant-zero trace, so
$(1)\Rightarrow(2)\Rightarrow(3)\Rightarrow(4)$.  Suppose that
$B=J_{P,S}^X$ is computable, where $S=\{s_0<s_1<\cdots\}$.  Define a total
computable predictor $Q$ on scheduled lengths by
\[
 Q(\sigma)=P(\sigma)\oplus B(k)
 \qquad\text{when }|\sigma|=s_k,
\]
and define it arbitrarily elsewhere.  Then
$Q(X\upharpoonright s_k)=X(s_k)$ for every $k$, so $S$ is a predictable
reservoir.  Thus $(4)\Rightarrow(1)$.

Finally, suppose first that $X$ is computable.  Then the full schedule is a
predictable reservoir, witnessed by the computable rule
$Q(\sigma)=X(|\sigma|)$.  Now suppose that $X$ is noncomputable.  By
Lemma~\ref{lem:two-sided-infinitude}, both $C_P(X)$ and $E_P(X)$ are infinite
for every total computable predictor $P$.  Hence, if $X$ is not
reservoir-immune, then for some $P$ one of these sets contains an infinite
c.e. subset and therefore an infinite computable subset.  In the second case
replace $P$ by its complement.  This gives a predictable reservoir.  The
converse is immediate.
\end{proof}

\begin{corollary}[The sequential refinement is proper]
\label{cor:sir-proper}
The relation $\lesir$ is a proper refinement of truth-table reducibility.
Indeed,
\[
 \zeroSeq\leq_{tt}X
\]
for every $X$, whereas
\[
 \zeroSeq\lesir X
 \quad\Longleftrightarrow\quad
 X\text{ is not reservoir-immune}.
\]
In particular, if $R$ is $\MartinLof$ random, then
$\zeroSeq\leq_{tt}R$ but $\zeroSeq\nleq_{\mathrm{sir}} R$.
\end{corollary}

\begin{proposition}[Randomness implies reservoir immunity]
\label{prop:random-reservoir-immune}
Every \MartinLof-random real is reservoir-immune.
\end{proposition}

\begin{proof}
For each total computable predictor $P$, the full innovation transform
$X\mapsto I_P^X$ is a computable measure-preserving bijection of Cantor space.
Hence $I_P^R$ is $\MartinLof$ random whenever $R$ is.  Every \MartinLof-
random real is bi-immune: if an infinite c.e. set $W$ were contained in the
real, then requiring successively more elements of $W$ to be $1$ would give an
effective null test; the same argument applies to the complement.  Therefore
$I_P^R$ is bi-immune for every $P$.
\end{proof}

\begin{remark}[Ordinary bi-immunity is insufficient]
Reservoir immunity is strictly stronger than ordinary bi-immunity.  Let $B$ be
bi-immune and define
\[
 X(2n)=X(2n+1)=B(n).
\]
Then $X$ is bi-immune, since projecting an infinite c.e. subset of $X$ or its
complement under $m\mapsto\lfloor m/2\rfloor$ would produce an infinite c.e.
subset of $B$ or its complement.  Nevertheless, the odd positions form a
predictable reservoir: the predictor that copies the preceding bit is correct
at every odd coordinate.  Thus ordinary bi-immunity blocks reservoirs for the
constant predictors, but not reservoirs created by sequential dependence on the
past.
\end{remark}

\begin{theorem}[Perfect RI Cantor sets inside effectively null classes]
\label{thm:null-cantor-ri}
Every nonempty basic cylinder $[\rho]\subseteq2^\omega$ contains a perfect
Cantor set
\[
\mathcal C\subseteq \mathrm{RI}\setminus\MLR.
\]
More precisely, one may choose a binary family of nonempty clopen sets
$(C_\tau)_{\tau\in2^{<\omega}}$ and a uniformly effectively open sequence
$(V_s)_{s\in\omega}$ such that:
\begin{enumerate}[label=(\roman*)]
\item $C_\varnothing\subseteq[\rho]\cap V_0.$
\item $C_{\tau0}$ and $C_{\tau1}$ are disjoint subsets of $C_\tau$.
\item for every $|\tau|=s>0$,
\[
 C_\tau\subseteq V_s,
 \qquad
 \lambda(C_\tau)\leq 2^{-2s}.
\]
\item $\lambda(V_s)\leq 2^{-s}$ for every $s$;
\item every member of
\[
 \mathcal C=\bigcap_{s\in\omega}\bigcup_{|\tau|=s}C_\tau
\]
is reservoir-immune.
\end{enumerate}
Consequently $(V_s)$ is a Martin--L\"of test covering $\mathcal C$, so no
member of $\mathcal C$ is Martin--L\"of random. No effectiveness is asserted for the fusion tree
$\tau\mapsto C_\tau$; the effective nullity refers to the uniformly
effectively open Martin--L\"of test $(V_s)$ covering $\mathcal C$.
\end{theorem}

\begin{proof}
Fix a uniformly effectively open sequence $(V_s)_{s\in\omega}$ such that
each $V_s$ is dense and
\begin{equation}
\lambda(V_s)\leq 2^{-s}.
\end{equation}
Such a sequence can be constructed uniformly as a union of basic clopen
cylinders. Enumerate all finite binary strings as
$\eta_0,\eta_1,\ldots$. For each $s$ and $j$, choose effectively an
extension $\widehat\eta_{s,j}\succeq\eta_j$ sufficiently long that
\begin{equation}
\lambda([\widehat\eta_{s,j}])
\leq 2^{-s-j-1},
\end{equation}
and put
\begin{equation}
V_s=\bigcup_{j\in\omega}[\widehat\eta_{s,j}].
\end{equation}
Every basic cylinder contains one of the cylinders
$[\widehat\eta_{s,j}]$, so $V_s$ is dense. Moreover,
\begin{equation}
\lambda(V_s)
\leq
\sum_{j\in\omega}2^{-s-j-1}
=
2^{-s}.
\end{equation}
Thus $(V_s)_{s\in\omega}$ is a Martin--L\"of test whose components are
dense.

Fix an enumeration $(P_i)_{i\in\omega}$ of the total computable predictors
and an effective enumeration $(W_e)_{e\in\omega}$ of the c.e. subsets of
$\omega$. The enumeration of the total computable predictors is used only
for this existence argument and need not itself be effective.

For every $i,e\in\omega$ and $b\in\{0,1\}$, consider the requirement
\begin{equation}
\mathcal R_{i,e}^b:\qquad
W_e\text{ infinite}
\Longrightarrow
(\exists n\in W_e)\ I_{P_i}^X(n)=b.
\end{equation}
If $W_e$ is infinite, define
\begin{equation}
U_{i,e}^b
=
\bigcup
\left\{
[\sigma c]:
|\sigma|=n\in W_e
\text{ and }
c=P_i(\sigma)\oplus b
\right\}.
\end{equation}
We claim that $U_{i,e}^b$ is open and dense. Openness is immediate. To
prove density, let $[\eta]$ be any basic cylinder. Since $W_e$ is infinite,
choose $n\in W_e$ with $n\geq|\eta|$, and extend $\eta$ to a string
$\sigma\in2^n$. Put
\begin{equation}
c=P_i(\sigma)\oplus b.
\end{equation}
Then $[\sigma c]\subseteq[\eta]\cap U_{i,e}^b$, and every
$X\in[\sigma c]$ satisfies
\begin{equation}
\begin{aligned}
I_{P_i}^X(n)
&=X(n)\oplus P_i(X\upharpoonright n)\\
&=c\oplus P_i(\sigma)\\
&=b.
\end{aligned}
\end{equation}
Thus $U_{i,e}^b$ is dense. If $W_e$ is finite, put
$U_{i,e}^b=2^\omega$. In either case, $U_{i,e}^b$ is a dense open set.

Choose a, not necessarily effective, enumeration $(U_s)_{s\in\omega}$ of
all the sets $U_{i,e}^b$. We now construct the family
$(C_\tau)_{\tau\in2^{<\omega}}$ by simultaneous fusion over all nodes at
each finite level.

Since $V_0$ is dense and open, the intersection $[\rho]\cap V_0$ is a
nonempty open set. Choose a nonempty basic cylinder
\begin{equation}
C_\varnothing\subseteq[\rho]\cap V_0.
\end{equation}

Suppose that, for some $s\in\omega$, pairwise disjoint nonempty basic
cylinders $C_\tau$ have been chosen for all $\tau\in2^s$. Fix one such
$\tau$. Since $U_s$ and $V_{s+1}$ are dense and open, the set
\begin{equation}
D_\tau=C_\tau\cap U_s\cap V_{s+1}
\end{equation}
is a nonempty open subset of Cantor space. We now justify carefully the
choice of the two children of $C_\tau$.

Because $D_\tau$ is nonempty and open, it contains a basic cylinder
$[\gamma_\tau]$. Choose an integer $N_\tau>|\gamma_\tau|$ sufficiently
large that
\begin{equation}
2^{-N_\tau}\leq 2^{-2(s+1)}.
\end{equation}
There are at least two distinct strings of length $N_\tau$ extending
$\gamma_\tau$. Choose two such incompatible extensions
$\gamma_{\tau,0}$ and $\gamma_{\tau,1}$ and define
\begin{equation}
C_{\tau0}=[\gamma_{\tau,0}],
\qquad
C_{\tau1}=[\gamma_{\tau,1}].
\end{equation}
Then $C_{\tau0}$ and $C_{\tau1}$ are nonempty disjoint basic cylinders and
\begin{equation}
C_{\tau j}
\subseteq
[\gamma_\tau]
\subseteq
D_\tau
\subseteq
C_\tau\cap U_s\cap V_{s+1}
\qquad (j=0,1).
\end{equation}
Furthermore,
\begin{equation}
\lambda(C_{\tau j})
=
2^{-N_\tau}
\leq
2^{-2(s+1)}
\qquad (j=0,1).
\end{equation}
Thus the finite requirements imposed at this stage do not obstruct
splitting. They merely restrict us to the nonempty open set $D_\tau$,
which still contains a basic cylinder and hence contains arbitrarily long
pairs of incompatible subcylinders. Increasing their lengths also gives
the required exponential measure bound.

Perform this choice for every one of the finitely many nodes
$\tau\in2^s$. This completes the construction of level $s+1$. Because
each child is contained in its parent, all requirements imposed at earlier
levels remain satisfied. Because each child is contained in $U_s$, the
requirement represented by $U_s$ is satisfied on every branch surviving
level $s+1$. Finally, every child lies in $V_{s+1}$ and satisfies the
prescribed measure estimate.

Define
\begin{equation}
F_s=\bigcup_{|\tau|=s}C_\tau
\end{equation}
and
\begin{equation}
\mathcal C=\bigcap_{s\in\omega}F_s.
\end{equation}
Each $F_s$ is a nonempty compact clopen set, and
$F_{s+1}\subseteq F_s$. Hence $\mathcal C$ is nonempty and closed. Along
each branch, the lengths of the defining cylinders increase strictly, so
every branch determines a unique real. Conversely, every real in
$\mathcal C$ determines a unique branch through the binary fusion tree.
Since every node has two incompatible children, no member of $\mathcal C$
is isolated. Therefore $\mathcal C$ is a perfect set, and the branch map
from $2^\omega$ onto $\mathcal C$ is a homeomorphism. Thus $\mathcal C$ is
a Cantor set.

For every $s>0$, there are exactly $2^s$ cylinders at level $s$, and each
has measure at most $2^{-2s}$. Therefore
\begin{equation}
\lambda(F_s)
\leq
2^s2^{-2s}
=
2^{-s}.
\end{equation}
It follows that
\begin{equation}
\lambda(\mathcal C)=0.
\end{equation}
More importantly, every level-$s$ cylinder is contained in $V_s$, and
hence
\begin{equation}
\mathcal C\subseteq\bigcap_{s\in\omega}V_s.
\end{equation}
Since $(V_s)_{s\in\omega}$ is a Martin--L\"of test, no member of
$\mathcal C$ is Martin--L\"of random. Consequently,
\begin{equation}
\mathcal C\cap\MLR=\varnothing.
\end{equation}

It remains to prove reservoir immunity. Fix $X\in\mathcal C$, a total
computable predictor $P_i$, and an infinite c.e. set $W_e$. For each
$b\in\{0,1\}$, the dense open set $U_{i,e}^b$ occurs as some $U_s$ in the
enumeration. All branches surviving level $s+1$, and therefore $X$, belong
to $U_{i,e}^b$. By the definition of this set, there is some $n\in W_e$
such that
\begin{equation}
I_{P_i}^X(n)=b.
\end{equation}
Thus both values $0$ and $1$ occur on $W_e$ in the innovation sequence
$I_{P_i}^X$. Hence no infinite c.e. set is contained in either the zero
set or the one set of $I_{P_i}^X$. Since this holds for every total
computable predictor $P_i$, every full innovation sequence $I_{P_i}^X$ is
bi-immune. Therefore $X$ is reservoir-immune.

We have shown that
\begin{equation}
\mathcal C\subseteq\mathrm{RI}\setminus\MLR,
\end{equation}
as required.
\end{proof}
\begin{theorem}[Predictable-reservoir theorem]
\label{thm:reservoir}
If $X$ has a predictable reservoir, then
\[
B\leq_{tt}X\quad\Longrightarrow\quad B\lesir X.
\]
Consequently,
\[
\Specinn(X)=\TTvis(X).
\]
\end{theorem}

\begin{proof}
Let $C=\{c_j\}$ be predicted correctly by $Q$, and suppose
\[
B(k)=g(k,X\upharpoonright u(k))
\]
for a total computable use $u$ and truth-table evaluator $g$. Choose a
computable strictly increasing sequence $j(k)$ such that
\[
s_k=c_{j(k)}\geq u(k).
\]
Define a total computable predictor $P$ at length $s_k$ by
\[
P(\sigma)=Q(\sigma)\oplus g(k,\sigma\upharpoonright u(k)),
\]
and define it arbitrarily at all other lengths. Then
\[
\begin{aligned}
J_{P,S}^X(k)
&=X(s_k)\oplus Q(X\upharpoonright s_k)
  \oplus g(k,X\upharpoonright u(k))\\
&=B(k).
\end{aligned}
\]
Thus $B\lesir X$. Combine this with
Proposition~\ref{prop:tt-bound}.
\end{proof}
\begin{remark}[Eventual reservoirs]
If $Q$ is correct on all but finitely many sites of $C$, the same construction
produces a finite variant of $B$. Therefore the equality of degree sets
$\Specinn(X)=\TTvis(X)$ still follows, even though exact equality of traces may
fail at finitely many coordinates.
\end{remark}

\begin{corollary}[Padded representatives]
\label{cor:padded}
For $X\in2^\omega$, define
\[
X^*(2n)=X(n),
\qquad
X^*(2n+1)=0.
\]
Then $X^*\equiv_{tt}X$, the odd positions form a predictable reservoir, and
\[
\Specinn(X^*)=\TTvis(X^*).
\]
\end{corollary}

\begin{remark}[The canonical presentation bridge]
If $X\in\mathrm{RI}$, then $X^*$ lies in the upper region,
$X\equiv_{tt}X^*$, and $X\lesir X^*$ by projection onto the even
coordinates. The reverse reduction is impossible, since it would place
$X$ above the computable SIR degree. Hence
\[
\deg_{\mathrm{sir}}(X)<\deg_{\mathrm{sir}}(X^*).
\]
Thus $X\mapsto X^*$ is the canonical bridge from the RI region to the
unique greatest SIR degree in the same truth-table degree.
\end{remark}

\begin{theorem}[The truth-table spine of the SIR degrees]
\label{thm:tt-spine}
Let
\[
\mathcal U_{\mathrm{sir}}
=\{\deg_{\mathrm{sir}}(X):\zeroSeq\lesir X\}.
\]
Then
\[
\pi:\mathcal U_{\mathrm{sir}}\longrightarrow\mathcal D_{tt},
\qquad
\pi(\deg_{\mathrm{sir}}(X))=\deg_{tt}(X),
\]
is an order isomorphism onto the truth-table degrees. Consequently:
\begin{enumerate}[label=(\roman*)]
\item $\mathcal U_{\mathrm{sir}}$ is an upper semilattice;
\item every truth-table degree contains a unique greatest SIR degree,
represented by any member of that truth-table degree having a predictable
reservoir;
\item if at least one of $X,Y$ has a predictable reservoir, then their join
exists in the full SIR degree structure and
\[
\deg_{\mathrm{sir}}(X)\vee\deg_{\mathrm{sir}}(Y)
=\deg_{\mathrm{sir}}(X\oplus Y).
\]
\end{enumerate}
\end{theorem}

\begin{proof}
The map $\pi$ is well defined because SIR equivalence implies truth-table
equivalence. If $X,Y$ lie in $\mathcal U_{\mathrm{sir}}$, then the target
$Y$ has a predictable reservoir, and therefore
\[
X\lesir Y\quad\Longleftrightarrow\quad X\leq_{tt}Y
\]
by Proposition~\ref{prop:tt-bound} and
Theorem~\ref{thm:reservoir}. Hence $\pi$ is an order embedding. It is
surjective because Corollary~\ref{cor:padded} supplies, in every truth-table
degree, a representative with a predictable reservoir.

Fix a truth-table degree $\mathbf d$ and choose a reservoir-bearing
representative $Y\in\mathbf d$. Every other $X\in\mathbf d$ satisfies
$X\leq_{tt}Y$, hence $X\lesir Y$. Thus the spine degree represented by $Y$
is the unique greatest SIR degree inside $\mathbf d$.

Finally, $X\oplus Y$ is always an SIR upper bound of $X$ and $Y$, witnessed
by the even and odd coordinate projections. Let $Z$ be any common SIR upper
bound, and suppose without loss of generality that $X$ has a predictable
reservoir. Then
\[
\zeroSeq\lesir X\lesir Z,
\]
so $Z$ also has a predictable reservoir. Since $X,Y\leq_{tt}Z$,
truth-table closure under effective disjoint union gives
$X\oplus Y\leq_{tt}Z$. The predictable-reservoir theorem now yields
$X\oplus Y\lesir Z$. Hence $X\oplus Y$ is the least common SIR upper bound.
\end{proof}

\begin{corollary}[C.e. degrees on the truth-table spine]
\label{cor:ce-principal-ideal}
Let $A$ be computably enumerable. Then $A$ has a predictable reservoir and
\[
\{B:B\lesir A\}=\{B:B\leq_{tt}A\}.
\]
Thus the computable SIR degree is the least SIR degree represented by a c.e.
real. Moreover:
\begin{enumerate}[label=(\roman*)]
\item if $A$ and $B$ are c.e. and $A\equiv_{tt}B$, then $A\equivsir B$;
\item the c.e.-represented SIR degrees are closed under joins, with
\[
\deg_{\mathrm{sir}}(A)\vee\deg_{\mathrm{sir}}(B)
=\deg_{\mathrm{sir}}(A\oplus B).
\]
\end{enumerate}
After passing to truth-table degrees, the degrees represented below $A$ form
the principal truth-table ideal below $\deg_{tt}(A)$.
\end{corollary}

\begin{proof}
If $A$ is finite, then it is computable and the full schedule is a predictable
reservoir. If $A$ is infinite, it contains an infinite computable subset
$S$. On $S$ the constant-one predictor is correct, so $S$ is a predictable
reservoir. The equality of the two classes of presentations follows from
Theorem~\ref{thm:reservoir} and Proposition~\ref{prop:tt-bound}. After
passing to truth-table degrees, this equality yields the principal
truth-table ideal below $\deg_{tt}(A)$. The remaining assertions follow from
Theorem~\ref{thm:tt-spine}; note that $A\oplus B$ is again c.e.
\end{proof}

\begin{corollary}[The exhaustive spine--RI partition]
\label{cor:two-canonical-regions}
Let
\[
\begin{aligned}
\mathcal U_{\mathrm{sir}}
  &=\{\mathbf a:\mathbf0_{\mathrm{sir}}\leq_{\mathrm{sir}}\mathbf a\},\\
\mathcal I_{\mathrm{RI}}
  &=\{\deg_{\mathrm{sir}}(X):X\in\mathrm{RI}\},\\
\mathcal R_{\mathrm{sir}}
  &=\{\deg_{\mathrm{sir}}(R):R\in\MLR\}.
\end{aligned}
\]
Then
\[
\mathcal D_{\mathrm{sir}}
=\mathcal U_{\mathrm{sir}}\mathbin{\dot\cup}\mathcal I_{\mathrm{RI}}.
\]
The upper region $\mathcal U_{\mathrm{sir}}$ is an upward-closed upper
semilattice isomorphic to the truth-table degrees, whereas
$\mathcal I_{\mathrm{RI}}$ is downward closed. Moreover,
\[
\mathcal R_{\mathrm{sir}}\subsetneq\mathcal I_{\mathrm{RI}},
\]
and $\mathcal R_{\mathrm{sir}}$ is itself downward closed. The strictness is
local and perfect: every basic cylinder contains a perfect Cantor set in
$\mathrm{RI}\setminus\MLR$.
\end{corollary}

\begin{proof}
By Proposition~\ref{prop:computable-trace-reservoir}, a real lies above the
computable SIR degree exactly when it is not reservoir-immune. This proves
the exhaustive disjoint partition. Upward closure and the semilattice
structure of $\mathcal U_{\mathrm{sir}}$ follow from
Theorem~\ref{thm:tt-spine}. If $X\in\mathrm{RI}$ and $Y\lesir X$, then
$Y\notin\mathcal U_{\mathrm{sir}}$, for otherwise
$\zeroSeq\lesir Y\lesir X$; hence $Y\in\mathrm{RI}$. Thus the RI region is
downward closed. Downward closure of $\mathcal R_{\mathrm{sir}}$ is
Corollary~\ref{cor:random-sir-lower-region}, and
Theorem~\ref{thm:null-cantor-ri} proves that the inclusion into the RI region
is proper in every basic cylinder.
\end{proof}

\begin{theorem}[Vanishing-density predictable reservoirs]
\label{thm:zero-density-reservoir}
Let $C\subseteq\omega$ be an infinite computable set of asymptotic density
zero, and enumerate its complement as $D=\{d_0<d_1<\cdots\}$.  For
$X\in2^\omega$, define the $C$-padded presentation $Y$ by
\[
        Y(n)=0\quad(n\in C),
        \qquad
        Y(d_k)=X(k).
\]
Then
\[
        Y\equiv_{tt}X,
        \qquad
        \Specinn(Y)=\TTvis(Y).
\]
Moreover, the lower and upper effective dimensions are preserved:
\[
        \dim_{\mathrm{eff}}(Y)=\dim_{\mathrm{eff}}(X),
        \qquad
        \operatorname{Dim}_{\mathrm{eff}}(Y)
        =\operatorname{Dim}_{\mathrm{eff}}(X).
\]
\end{theorem}

\begin{proof}
The computable zero set $C$ is a predictable reservoir, so
Theorem~\ref{thm:reservoir} gives
$\Specinn(Y)=\TTvis(Y)$.  Deleting the coordinates in $C$ recovers $X$, while
inserting zeroes on $C$ computes $Y$ from $X$, hence $Y\equiv_{tt}X$.

Let
\[
        m(n)=|D\cap n|=n-|C\cap n|.
\]
Then $m(n)=n-o(n)$.  Uniformly in $n$, the prefix $Y\upharpoonright n$ is
computable from $X\upharpoonright m(n)$ and $n$, while
$X\upharpoonright m(n)$ is computable from $Y\upharpoonright n$.  Therefore
\[
 K(X\upharpoonright m(n))-O(\log n)
 \leq K(Y\upharpoonright n)
 \leq K(X\upharpoonright m(n))+O(\log n).
\]
Since $m(n)/n\to1$ and $m$ assumes every natural value, taking lower and upper
limits proves both dimension equalities.
\end{proof}

\begin{remark}[Increasing-block implementation]
One may partition $\omega$ into computable blocks of increasing lengths and
reserve only the last coordinate of each block for the value $0$.  The
reserved coordinates then form an infinite density-zero reservoir.  Thus the
permanent even--odd empty column in Corollary~\ref{cor:padded} is a convenient
normal form, not a density requirement.  Arbitrarily sparse predictable sites
suffice because they may be placed after the use of each desired truth-table
computation.
\end{remark}

\begin{corollary}[The spectrum is not a truth-table-degree invariant]
\label{cor:not-degree-invariant}
There are $X\equiv_{tt}Y$ with
\[
        \Specinn(X)\neq\Specinn(Y).
\]
\end{corollary}

\begin{proof}
Let $R$ be $\MartinLof$ random and let $R^*$ be its padded representative.  Every
computable trace of $R$ is $\MartinLof$ random by
Theorem~\ref{thm:residual-ls}, so the computable degree is absent from
$\Specinn(R)$.  The odd zero column of $R^*$ is a computable trace, so
$\mathbf 0\in\Specinn(R^*)$.  Yet $R\equiv_{tt}R^*$.
\end{proof}

\begin{theorem}[Two independent collapse criteria]
\label{thm:three-visibility-collapse}
If $X$ has a predictable reservoir and $\emptyset'\leq_{tt}X$, then
\[
 \Specinn(X)=\TTvis(X)=\WTTvis(X).
\]
\end{theorem}

\begin{proof}
The predictable-reservoir theorem gives
$\Specinn(X)=\TTvis(X)$, while
Remark~\ref{rem:wtt-totalization} gives
$\TTvis(X)=\WTTvis(X)$.
\end{proof}

\begin{remark}[Qualitative versus quantitative collapse]
Theorem~\ref{thm:three-visibility-collapse} concerns represented Turing
degrees.  It does not preserve use functions or schedule density.  Totalizing
a weak truth-table computation may inspect all counterfactual strings of a
given bounded length, and realizing the resulting truth table as a trace may
delay the output to a much later reservoir coordinate.  Thus entropy and rate
questions remain strictly finer than equality of visible degree sets.
\end{remark}

The preceding results describe the upper part of the SIR degree structure.  We now turn to its lower boundary, where reservoir immunity prevents the computable reals from playing the role of a least degree.

\begin{proposition}[The computable degree is minimal but not least]
\label{prop:sir-minimal-not-least}
All computable reals belong to one sequential-innovation degree, and this degree is
minimal.  It is not a least degree.  More precisely,
\[
 \zeroSeq\lesir X
 \quad\Longleftrightarrow\quad
 X\text{ is not reservoir-immune}.
\]
For every $X$, the visible set $\Specinn(X)$ has the top element $\deg_T(X)$,
but it may omit $\mathbf0$ and need not be a Turing lower ideal.
\end{proposition}

\begin{proof}
Computable reals predict one another by hard-coding the desired output, and
every trace of a computable source is computable.  Hence their common degree is
minimal.  Proposition~\ref{prop:computable-trace-reservoir} gives the displayed
equivalence and shows that no reservoir-immune source lies above the computable
sequential-innovation degree.  The trivial full-schedule channel places
$\deg_T(X)$ in $\Specinn(X)$.  If $R$ is $\MartinLof$ random, then
$\mathbf0\notin\Specinn(R)$ although $\deg_T(R)$ belongs to the spectrum, so
the spectrum is not a Turing lower ideal.
\end{proof}

\begin{corollary}[A truth-table degree splits into sequential-innovation degrees]
\label{cor:tt-degree-splits-sir}
Let $R$ be $\MartinLof$ random and let $R^*$ be any computably padded
presentation with an infinite predictable reservoir.  Then
\[
 R\equiv_{tt}R^*,
 \qquad
 R\not\equiv_{\mathrm{sir}}R^*.
\]
Thus sequential-innovation equivalence is a proper refinement of truth-table
equivalence even within a single truth-table degree.
\end{corollary}

\begin{proof}
The padding and deletion maps are truth-table reductions.  If
$R\equiv_{\mathrm{sir}}R^*$, then the computable trace available from $R^*$
would, by transitivity, also be available from $R$, contradicting
Proposition~\ref{prop:random-reservoir-immune}.
\end{proof}

The preceding statements explain the role of immunity.  Ordinary bi-immunity
blocks computable reservoirs visible to constant predictors, whereas reservoir
immunity blocks them for every sequential computable predictor.  It is exactly the
obstruction preventing the computable sequential-innovation degree from lying below
a presentation.

\section{SIR spectra of random, Bernoulli, and c.e. sources}

By Corollary~\ref{cor:random-sir-lower-region}, the fair-coin random region is
closed downward under SIR: no computable sequential procedure can derandomize
a Martin--L\"of-random source.  This rigidity coexists with extensive internal
order structure created by computable coordinate geometry.  We first separate
finite causal delay inside one Turing degree from genuine information loss
across Turing degrees.

For computable $p\in(0,1)$, let $\mu_p$ denote the Bernoulli
product measure with one-bit marginal $\mu_p\{1\}=p$.

\begin{lemma}[Bernoulli column independence]
\label{lem:bernoulli-column-independence}
Let $p\in(0,1)$ be computable, let $R$ be $\mu_p$-Martin--L\"of
random, and let $U,V\subseteq\omega$ be disjoint infinite computable sets.
Then
\[
        R\upharpoonright U
        \in \MLR_{\mu_p}^{R\upharpoonright V}.
\]
In particular, $R\upharpoonright U\nleq_T R\upharpoonright V$.
\end{lemma}

\begin{proof}
Enumerate $U$ and $V$ increasingly.  The coordinate map
\[
        \Theta_{U,V}(R)
        =\bigl(R\upharpoonright U,R\upharpoonright V\bigr)
\]
is computable and pushes $\mu_p$ forward to the product measure
$\mu_p\times\mu_p$.  Randomness conservation for computable
measure-preserving maps therefore makes $\Theta_{U,V}(R)$
$\mu_p\times\mu_p$-Martin--L\"of random
\cite{HoyrupRojas2009}.  After interchanging the two coordinates, the
product-measure form of van Lambalgen's theorem gives
$R\upharpoonright U\in\MLR_{\mu_p}^{R\upharpoonright V}$; here the
conditional measure is the fixed computable measure $\mu_p$
\cite{vanLambalgen1990,Takahashi2011,Nies2009}.  Since $\mu_p$ is
nonatomic, no sequence computable from the oracle
$R\upharpoonright V$ can be $\mu_p$-Martin--L\"of random relative to that
oracle.  Hence $R\upharpoonright U\nleq_T R\upharpoonright V$.
\end{proof}
\begin{theorem}[Bernoulli-random universality]
\label{thm:bernoulli-poset-universality}
Suppose that $p\in(0,1)$ is computable and that $X$ has a
$\mu_p$-Martin--L\"of-random sequential innovation trace $R\lesir X$.
Then every computable countable partial order $(\mathcal A,\preceq)$ admits
an order embedding
\begin{equation}
        a\longmapsto \deg_{\mathrm{sir}}(R_a)
\end{equation}
into the SIR degrees strictly below $\deg_{\mathrm{sir}}(X)$, where every
$R_a$ is $\mu_p$-Martin--L\"of random. Consequently, below $X$ there are
$\mu_p$-random copies of every computable tree, every computable countable
linear order, countably infinite antichains, and strictly increasing chains
of every computable ordinal type $\alpha<\omega_1^{CK}$.
\end{theorem}

\begin{proof}
Assume that $\mathcal A$ is a computable subset of $\omega$ and that
$\preceq$ is computable. For $a\in\mathcal A$, put
\begin{equation}
        D_a=\{b\in\mathcal A:b\preceq a\},
        \qquad
        \widehat D_a=\{2b:b\in D_a\}.
\end{equation}
Each $\widehat D_a$ is a nonempty computable set, and reflexivity gives
\begin{equation}
        \widehat D_a\subseteq\widehat D_c
        \quad\Longleftrightarrow\quad
        a\preceq c.
\end{equation}
Partition the coordinates of $R$ into canonical computable infinite columns
$(C_i)_{i\in\omega}$ and define
\begin{equation}
        S_a=\bigcup_{i\in\widehat D_a}C_i,
        \qquad
        R_a=R\upharpoonright S_a.
\end{equation}
Every $R_a$ is $\mu_p$-Martin--L\"of random by randomness conservation
under the computable coordinate projection to $S_a$, and it is a sequential
innovation trace of $R$, hence of $X$.

If $a\preceq c$, computable coordinate projection gives $R_a\lesir R_c$.
Conversely, if $a\npreceq c$, then
$2a\in\widehat D_a\setminus\widehat D_c$. Apply
Lemma~\ref{lem:bernoulli-column-independence} with $U=C_{2a}$ and $V=S_c$.
It gives
\begin{equation}
        R\upharpoonright C_{2a}\nleq_T R_c.
\end{equation}
The left-hand side is computable from $R_a$, so $R_a\nleq_T R_c$, and
therefore $R_a\nlesir R_c$.

Finally, every $S_a$ omits the infinite computable union $O$ of all
odd-indexed columns. Lemma~\ref{lem:bernoulli-column-independence}, applied
to $U=O$ and $V=S_a$, gives
\begin{equation}
        R\upharpoonright O\nleq_T R_a.
\end{equation}
Since $R\upharpoonright O\leq_T R$, this implies $R\nleq_T R_a$, and hence
$R_a<_T R$. If $X\lesir R_a$, then
$R\lesir X\lesir R_a$, contradicting $R\nleq_T R_a$. Thus every embedded
degree is strictly below $\deg_{\mathrm{sir}}(X)$.
\end{proof}

For $k<\omega$, write
\[
        \sigma^kX(n)=X(n+k)
\]
for the $k$-fold left shift.

\begin{lemma}[Finite variants and shifted innovations]
\label{lem:finite-variants-shifted-innovations}
\begin{enumerate}[label=(\roman*)]
\item If $X$ and $Y$ differ on only finitely many coordinates, then
$X\equivsir Y$.
\item For every total computable predictor $P$ and every $k<\omega$,
\[
        \sigma^k I_P^X\equivsir\sigma^kX.
\]
\end{enumerate}
\end{lemma}

\begin{proof}
For~(i), use the full schedule and hardwire the finitely many values of
$X(n)\oplus Y(n)$ into the predictor.  The same construction works in both
directions.

For~(ii), hardwire $X\upharpoonright k$.  From a prefix of $\sigma^kX$ one
can reconstruct the corresponding prefix of $X$ and apply $P$, giving
$\sigma^k I_P^X\lesir\sigma^kX$.  Conversely, from a prefix of
$\sigma^k I_P^X$, recursively reconstruct the corresponding source prefix,
again using the hardwired initial segment.  This gives the reverse full-schedule
SIR reduction.
\end{proof}

\begin{theorem}[Random interval dichotomy]
\label{thm:random-interval-dichotomy}
Let $B$ be Martin--L\"of random and let $A\lesir B$.
Exactly one of the following alternatives holds.
\begin{enumerate}[label=(\roman*)]
\item $A\equiv_TB$.  Then there is a unique $k<\omega$ such that
\[
        A\equivsir\sigma^kB.
\]
The shifts form a strict chain
\[
\cdots<_{\mathrm{sir}}\sigma^2B
<_{\mathrm{sir}}\sigma B<_{\mathrm{sir}}B,
\]
and consecutive shifts are cover pairs.  More generally, the interval between
$\sigma^kB$ and $B$ consists exactly of the $k+1$ shift degrees
$\sigma^jB$, $0\leq j\leq k$.
\item $A<_TB$.  Then the open SIR interval
\[
\bigl(\deg_{\mathrm{sir}}(A),\deg_{\mathrm{sir}}(B)\bigr)
\]
contains an order-isomorphic copy of $(\mathbb Q,<)$ represented entirely by
Martin--L\"of-random reals.
\end{enumerate}
\end{theorem}

\begin{proof}
Choose a channel with
\[
        A=J_{P,S}^B
\]
after identifying $S$ with its increasing enumeration, and put
$I=I_P^B$.  By Proposition~\ref{prop:innovation-homeomorphism},
$I\equivsir B$ and $I$ is Martin--L\"of random, while
\[
        A=I\upharpoonright S.
\]

Suppose first that $A\equiv_TB$.  If $\omega\setminus S$ were infinite, then
van Lambalgen's theorem applied to the computable partition
$S\sqcup(\omega\setminus S)$ would make
$I\upharpoonright(\omega\setminus S)$ random relative to $A$.  It could not
be computed from $A$, contradicting $I\equiv_TA$.  Hence $S$ is cofinite.
If $k=|\omega\setminus S|$, then $A$ is a finite variant of $\sigma^kI$.
Lemma~\ref{lem:finite-variants-shifted-innovations} gives
$A\equivsir\sigma^kB$.

For every $k$, the tail projection gives
$\sigma^{k+1}B\lesir\sigma^kB$.  The reverse reduction is impossible.  Indeed,
put $X=\sigma^{k+1}B$ and suppose that
$\sigma^kB\lesir X$ via schedule $s_n$ and predictor $Q$.  For $n\geq1$,
\[
        X(n-1)=X(s_n)\oplus Q(X\upharpoonright s_n).
\]
Since $s_n\geq n$, define a total computable predictor $H$ at length
$s_n$, for $n\geq1$, by
\[
        H(\tau)=Q(\tau)\oplus\tau(n-1),
\]
where $n$ is recovered effectively from the scheduled length $s_n$, and define
$H$ arbitrarily elsewhere.  Then
$H(X\upharpoonright s_n)=X(s_n)$ on the infinite computable tail of the
schedule, contradicting reservoir immunity of the random real $X$.  Thus the
shift chain is strict.  If
$\sigma^kB\lesir Z\lesir B$, then $Z\equiv_TB$, so the preceding
classification gives $Z\equivsir\sigma^jB$ for a unique $j$.  The strict shift
order forces $0\leq j\leq k$, proving the cover and finite-interval claims.

Now suppose that $A<_TB$.  Then $D=\omega\setminus S$ is infinite.  Enumerate
$D=\{d_0<d_1<\cdots\}$, fix a computable one-to-one enumeration
$(r_i)_{i\in\omega}$ of $\mathbb Q\cap(0,1)$, and use a computable pairing
function.  For $q\in\mathbb Q\cap(0,1)$ define
\[
D_q=\{d_{\langle i,j\rangle}:r_i<q\},
\qquad
T_q=S\cup D_q,
\qquad
Z_q=I\upharpoonright T_q.
\]
Each $T_q$ is computable.  If $q<q'$, then
$D_q$, $D_{q'}\setminus D_q$, and $D\setminus D_{q'}$ are all infinite,
because the rationals are dense and each rational index carries an infinite
column.  Coordinate projection gives
\[
        A\lesir Z_q\lesir Z_{q'}\lesir I\equivsir B.
\]
Van Lambalgen's theorem applied to each newly added or omitted infinite column
shows that all three inequalities are strict already in Turing degree.  Every
$Z_q$ is Martin--L\"of random.  Hence
$q\mapsto\deg_{\mathrm{sir}}(Z_q)$ embeds
$(\mathbb Q\cap(0,1),<)$ into the open interval.
\end{proof}

\begin{theorem}[A random pair with no SIR join]
\label{thm:random-no-join}
Let $R$ be Martin--L\"of random and define its even and odd columns by
\[
        A(n)=R(2n),
        \qquad
        B(n)=R(2n+1).
\]
Then $A$ and $B$ are Martin--L\"of random, their SIR degrees are incomparable,
and they have no join in the full SIR degree structure.
\end{theorem}

\begin{proof}
The identity $R=A\oplus B$ and van Lambalgen's theorem show that $A$ and $B$
are Martin--L\"of random and Turing incomparable, hence SIR incomparable.

We first isolate the mechanism producing common upper bounds.  For every
$m<\omega$, put $X_m=\sigma^mR$.  Then both $A$ and $B$ are SIR-reducible
to $X_m$.  For $A$, choose $N>m$ and define a strictly increasing schedule by
\[
 s_n=n\quad(n<N),
 \qquad
 s_n=2n-m\quad(n\geq N).
\]
At the finitely many initial scheduled lengths, define the predictor so that
\[
 P(X_m\upharpoonright n)=X_m(n)\oplus A(n)
 \qquad(n<N),
\]
and extend these finitely prescribed values computably to all strings of the
corresponding lengths.  At every later scheduled length set $P$ equal to zero.
Since
\[
 X_m(2n-m)=R(2n)=A(n)
 \qquad(n\geq N),
\]
the resulting trace is exactly $A$.  Replacing the eventual schedule by
$s_n=2n+1-m$ gives $B\lesir X_m$.  Thus every finite tail
\[
 \sigma^mR
\]
is a common SIR upper bound of $A$ and $B$.

Suppose now that a join $J$ existed.  Since $R$ is one of these common upper
bounds, leastness gives $J\lesir R$.  Conversely, $A,B\lesir J$ implies
\[
 R=A\oplus B\leq_TJ.
\]
Hence $J\equiv_TR$.  The same-Turing-degree clause of
Theorem~\ref{thm:random-interval-dichotomy} therefore yields
\[
 J\equivsir\sigma^kR
\]
for some $k<\omega$.  But $\sigma^{k+1}R$ is also a common upper bound, so
leastness of $J$ gives
\[
 J\lesir\sigma^{k+1}R.
\]
Combining this with $\sigma^kR\lesir J$ gives
\[
 \sigma^kR\lesir\sigma^{k+1}R,
\]
contradicting the strict shift chain.  Thus no join exists.

The common-upper-bound construction itself uses only the interleaving
$R=A\oplus B$.  Reservoir immunity supplies the strictness of the shift
chain; Martin--L\"of randomness is used here to obtain incomparability and to
classify every same-Turing-degree SIR descendant of $R$ as a finite shift.
\end{proof}
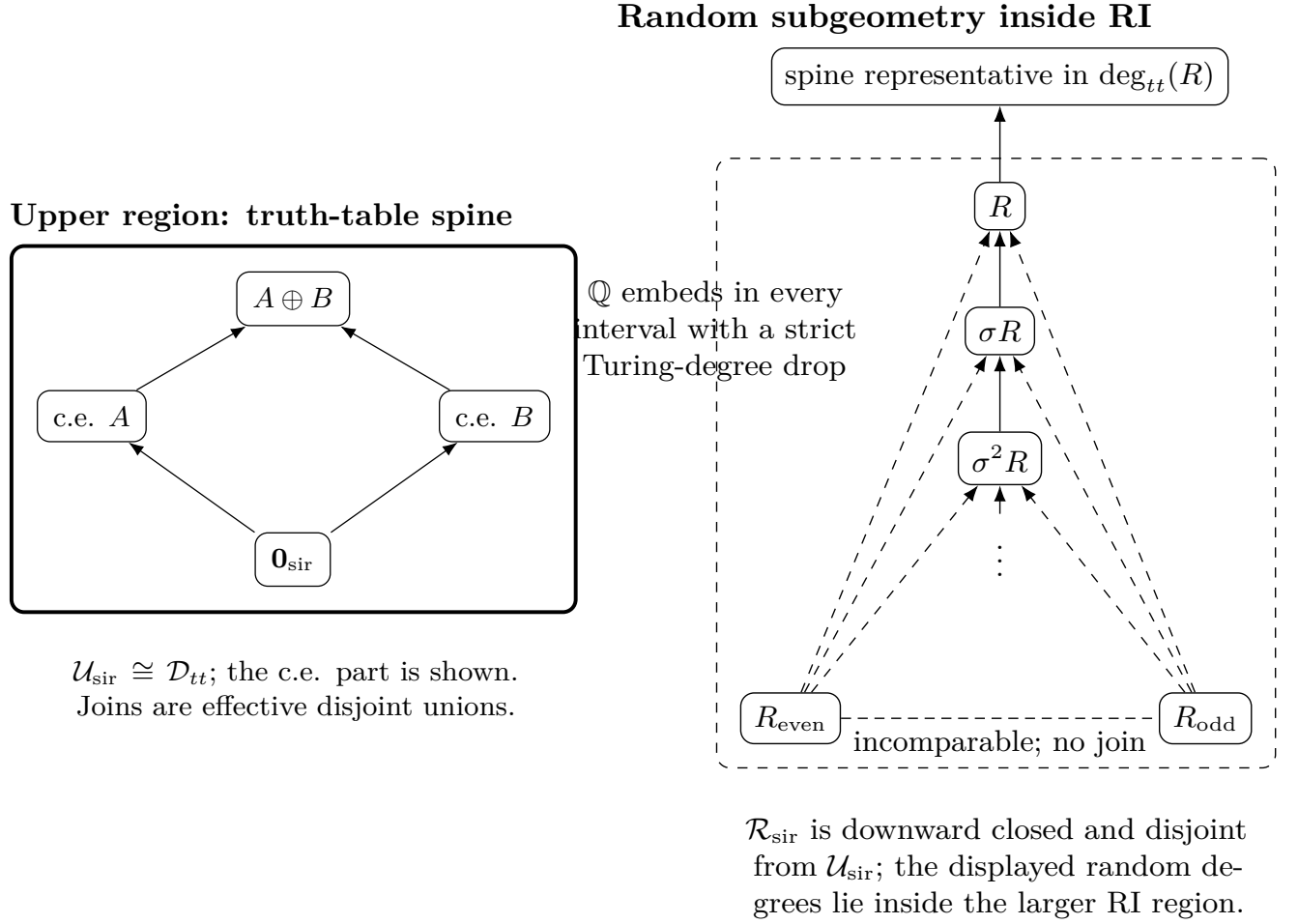
\begin{figure}[!t]
\centering
\begin{minipage}[t]{0.42\textwidth}
\centering
\resizebox{\linewidth}{!}{%
\begin{minipage}{56mm}
\centering
\textbf{Upper region: truth-table spine}\par\smallskip
\begin{tikzpicture}[
  >=Latex,
  node distance=10mm and 12mm,
  every node/.style={font=\small,align=center},
  degree/.style={draw,rounded corners,inner sep=5pt}
]
\node[degree] (z) {$\mathbf0_{\mathrm{sir}}$};
\node[degree,above left=of z] (a) {c.e. $A$};
\node[degree,above right=of z] (b) {c.e. $B$};
\node[degree,above=of $(a)!0.5!(b)$] (j) {$A\oplus B$};
\draw[->] (z) -- (a);
\draw[->] (z) -- (b);
\draw[->] (a) -- (j);
\draw[->] (b) -- (j);
\node[draw,rounded corners,very thick,fit=(z)(a)(b)(j),inner sep=8pt] (box) {};
\node[below=4mm of box,text width=56mm]
{$\mathcal U_{\mathrm{sir}}\cong\mathcal D_{tt}$; the c.e. part is shown.
Joins are effective disjoint unions.};
\end{tikzpicture}
\end{minipage}}
\end{minipage}
\hfill
\begin{minipage}[t]{0.54\textwidth}
\centering
\resizebox{\linewidth}{!}{%
\begin{minipage}{68mm}
\centering
\textbf{Random subgeometry inside RI}\par\smallskip
\begin{tikzpicture}[
  >=Latex,
  node distance=8mm and 18mm,
  every node/.style={font=\small,align=center},
  degree/.style={draw,rounded corners,inner sep=4pt}
]
\node[degree] (star) {spine representative in $\deg_{tt}(R)$};
\node[degree,below=of star] (r) {$R$};
\node[degree,below=of r] (s1) {$\sigma R$};
\node[degree,below=of s1] (s2) {$\sigma^2R$};
\node[below=3mm of s2] (dots) {$\vdots$};
\node[degree,below left=11mm and 15mm of dots] (e) {$R_{\rm even}$};
\node[degree,below right=11mm and 15mm of dots] (o) {$R_{\rm odd}$};
\draw[->] (r) -- (star);
\draw[->] (s1) -- (r);
\draw[->] (s2) -- (s1);
\draw[->] (dots) -- (s2);
\draw[->,dashed] (e) -- (r);
\draw[->,dashed] (o) -- (r);
\draw[->,dashed] (e) -- (s1);
\draw[->,dashed] (o) -- (s1);
\draw[->,dashed] (e) -- (s2);
\draw[->,dashed] (o) -- (s2);
\draw[densely dashed] (e) -- node[below] {incomparable; no join} (o);
\node[left=10mm of s1,text width=30mm] {$\mathbb Q$ embeds in every interval
with a strict Turing-degree drop};
\node[draw,dashed,rounded corners,fit=(r)(s1)(s2)(dots)(e)(o),inner sep=7pt] (rbox) {};
\node[below=4mm of rbox,text width=61mm]
{$\mathcal R_{\mathrm{sir}}$ is downward closed and disjoint from
$\mathcal U_{\mathrm{sir}}$; the displayed random degrees lie inside the larger RI region.};
\end{tikzpicture}
\end{minipage}}
\end{minipage}
\caption{The truth-table spine and the random subgeometry of the complementary RI region.  Reducibility arrows
point upward.  The upper cone above the computable degree is the truth-table
spine and supports joins.  Its complement is the downward-closed RI region;
the figure displays the fair-coin Martin--L\"of-random subclass, within which
every SIR output remains random.  A random presentation reduces to the unique
greatest SIR degree in its truth-table degree via the padding bridge
$R\mapsto R^*$.  Inside the random subgeometry, finite shifts form discrete
causal layers, intervals crossing Turing degrees contain copies of
$\mathbb Q$, and the even and odd columns of $R$ have common upper bounds but
no least one.}
\label{fig:sir-spine-random-region}
\end{figure}

\begin{problem}[Beyond Martin--L\"of randomness]
\label{prob:reservoir-immune-structure}
Determine which parts of Theorems~\ref{thm:random-interval-dichotomy} and
\ref{thm:random-no-join} extend from Martin--L\"of-random presentations to
arbitrary reservoir-immune presentations.  The proofs above use van
Lambalgen's theorem to ensure that every omitted infinite computable column
remains noncomputable relative to the retained columns.  Reservoir immunity
alone does not presently supply an equivalent relative-independence principle.
\end{problem}

\begin{corollary}[Spectral barrenness under thin truth-table visibility]
\label{cor:thin-cone-no-random-trace}
If the truth-table-visible degree set of $X$ contains no two incomparable noncomputable
Turing degrees, then $X$ has no \MartinLof-random sequential innovation trace.  In
particular, no real of minimal nonzero Turing degree has a \MartinLof-random
trace.
\end{corollary}

\begin{proof}
Every sequential innovation trace is truth-table reducible to $X$.  If one such trace were
$\MartinLof$ random, Theorem~\ref{thm:bernoulli-poset-universality}, with $p=1/2$, would produce
two incomparable noncomputable degrees in the truth-table-visible degree set of $X$.
\end{proof}

\begin{remark}
The conclusion is stronger than nonrandomness of the presentation itself.  A
spectrally barren real admits no computable sequential observation whose normalized
innovation sequence is $\MartinLof$ random.  Conversely, a single random
trace can never be an isolated witness: it automatically generates extensive
random degree structure below the original presentation.
\end{remark}

If $R$ is $\MartinLof$ random, every member of $\Specinn(R)$ is represented by
a \MartinLof-random real.  In particular,
\[
        \mathbf0\notin\Specinn(R).
\]
This complexity rigidity coexists with substantial degree-theoretic structure.

Theorem~\ref{thm:bernoulli-poset-universality} subsumes the customary
branching examples.  In the fair-coin case it yields, strictly below every
random trace, copies of every computable tree and linear order, countable
antichains, dense copies of $(\mathbb Q,<)$, and chains of every computable
ordinal type.  For a general computable $p\in(0,1)$, the same structures are
represented by $\mu_p$-Martin--L\"of-random reals.

We conclude with a mixed presentation that places the c.e. truth-table spine
and a transverse random family below a single SIR degree.  Recall that a
$2$-random real is \MartinLof\ random relative to $\emptyset'$, and that
every $2$-random degree forms a minimal pair with $\mathbf0'$; see
\cite{Nies2009,DowneyHirschfeldt2010}.

Fix a computable partition of the coordinates of a $2$-random real $R$ into
infinite columns,
\begin{equation}
R=\bigoplus_{i\in\omega}R_i,
\end{equation}
and, for every nonempty computable $I\subseteq\omega$, put
\begin{equation}
R_I=\bigoplus_{i\in I}R_i.
\end{equation}

\begin{theorem}[The c.e.--random bridge]
\label{thm:ce-random-bridge}
Let $R$ be $2$-random and put $X=\emptyset'\oplus R$. Then:

\begin{enumerate}[label=(\roman*)]
\item
\begin{equation}
\Specinn(X)=\TTvis(X)=\WTTvis(X),
\end{equation}
and every c.e. Turing degree belongs to $\Specinn(X)$.

\item For all nonempty computable $I,J\subseteq\omega$, the reals $R_I$ and
$\emptyset'\oplus R_I$ belong to $\Specinn(X)$, and
\begin{equation}
R_I\leq_T R_J
\quad\Longleftrightarrow\quad
R_I\leq_T\emptyset'\oplus R_J
\quad\Longleftrightarrow\quad
\emptyset'\oplus R_I\leq_T\emptyset'\oplus R_J
\quad\Longleftrightarrow\quad
I\subseteq J.
\end{equation}

\item Every $R_I$ is $2$-random and is Turing incomparable with every
noncomputable c.e. set, whereas every $\emptyset'\oplus R_I$ lies above all
c.e. degrees and is not Turing reducible to any $R_J$.
\end{enumerate}
\end{theorem}

\begin{proof}
The $\emptyset'$ component of $X$ contains an infinite computable predictable
reservoir, obtained from indices of machines that halt immediately.  Hence
Theorem~\ref{thm:three-visibility-collapse} gives
\begin{equation}
\Specinn(X)=\TTvis(X)=\WTTvis(X).
\end{equation}
Every c.e. set is many-one reducible to $\emptyset'$, so every c.e. degree
belongs to this spectrum.

Each $R_I$ is a computable coordinate projection of a real that is random
relative to $\emptyset'$, and is therefore $2$-random.  If $I\subseteq J$,
the required reductions follow by coordinate projection.  If
$i\in I\setminus J$, relativized van Lambalgen gives
\begin{equation}
R_i\in\MLR^{\emptyset'\oplus R_J},
\end{equation}
and consequently
\begin{equation}
R_i\nleq_T\emptyset'\oplus R_J.
\end{equation}
Since $R_i\leq_T R_I$, none of the three reductions in part~(ii) can hold.
This proves the equivalences.  All the displayed reals are truth-table
reducible to $X$, so they belong to $\Specinn(X)$.

Finally, a $2$-random real is not computable from $\emptyset'$.  It therefore
cannot be computed from a c.e. set.  Conversely, if a noncomputable c.e. set
were computable from $R_I$, it would give a nonzero common lower bound of
$\deg_T(R_I)$ and $\mathbf0'$, contradicting their minimal-pair property.
The remaining assertions follow from
$\emptyset'\leq_T\emptyset'\oplus R_I$ and
$\emptyset'\nleq_T R_J$.
\end{proof}

At the level of SIR degrees, the c.e.-represented branch is ordered by
truth-table reducibility, since c.e. presentations have predictable
reservoirs.  Every $R_I$ is Martin--L\"of random and hence reservoir-immune.
It follows that the c.e.-represented branch and the random branch are
entirely SIR-incomparable: reducibility in either direction would contradict
either reservoir immunity or the Turing incomparability established above.
Figure~\ref{fig:ce-random-bridge} summarizes the resulting guaranteed
suborder below $\deg_{\mathrm{sir}}(\emptyset'\oplus R)$.

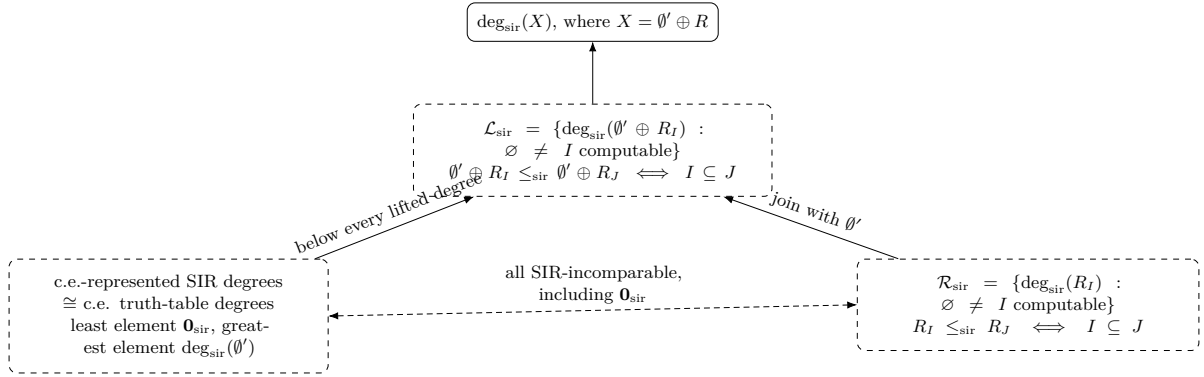
\begin{figure}[!htbp]
\centering
\resizebox{0.96\textwidth}{!}{%
\begin{tikzpicture}[
  >=Latex,
  node distance=11mm and 15mm,
  every node/.style={font=\small,align=center},
  degree/.style={draw,rounded corners,inner sep=5pt},
  family/.style={draw,dashed,rounded corners,inner sep=8pt,text width=58mm}
]
\node[degree] (top) {$\deg_{\mathrm{sir}}(X)$, where $X=\emptyset'\oplus R$};
\node[family,below=of top] (lifted)
  {$\mathcal L_{\mathrm{sir}}=
    \{\deg_{\mathrm{sir}}(\emptyset'\oplus R_I):
      \varnothing\neq I\text{ computable}\}$\\
   $\emptyset'\oplus R_I\lesir\emptyset'\oplus R_J
      \iff I\subseteq J$};
\node[family,below left=of lifted,text width=51mm] (ce)
  {c.e.-represented SIR degrees\\
   $\cong$ c.e. truth-table degrees\\
   least element $\mathbf0_{\mathrm{sir}}$, greatest element
   $\deg_{\mathrm{sir}}(\emptyset')$};
\node[family,below right=of lifted,text width=55mm] (random)
  {$\mathcal R_{\mathrm{sir}}=
    \{\deg_{\mathrm{sir}}(R_I):
      \varnothing\neq I\text{ computable}\}$\\
   $R_I\lesir R_J\iff I\subseteq J$};

\draw[->] (lifted) -- (top);
\draw[->] (ce) -- node[above,sloped] {below every lifted degree} (lifted);
\draw[->] (random) -- node[above,sloped] {join with $\emptyset'$} (lifted);
\draw[densely dashed,<->] (ce.east) --
  node[above,text width=31mm] {all SIR-incomparable, including
  $\mathbf0_{\mathrm{sir}}$} (random.west);
\end{tikzpicture}%
}
\caption{A guaranteed SIR suborder below
$\deg_{\mathrm{sir}}(\emptyset'\oplus R)$.  The c.e.-represented branch is
a copy of the c.e. truth-table degrees; the random branch is SIR-incomparable
with it; and adjoining $\emptyset'$ to a random node produces the
corresponding lifted node.  The full lower cone may contain additional
degrees.}
\label{fig:ce-random-bridge}
\end{figure}

As shown in Figure~\ref{fig:ce-random-bridge}, the lower SIR cone of
$\emptyset'\oplus R$ contains three interacting parts: the c.e.-represented
spine, ordered by truth-table reducibility; a transverse family of
reservoir-immune random degrees; and the corresponding lifted family above
$\deg_{\mathrm{sir}}(\emptyset')$.  The diagram records only this guaranteed
suborder; additional degrees may also occur.

\section{Application: fair-random traces from Bernoulli sources}
\label{sec:bernoulli-extraction}

The degree-theoretic results above describe the structure below a random
trace.  We now show that every presentation which is Martin--L\"of random
for a computable nondegenerate Bernoulli measure has a fair-coin
Martin--L\"of-random SIR trace.  Moreover, the trace can be obtained on a
computable schedule of positive asymptotic density.

For computable $p\in(0,1)$, let $\mu_p$ denote Bernoulli measure and put
\begin{equation}
q=2p(1-p).
\end{equation}
For two independent $p$-biased bits, the pair is unequal with probability
$q$, and conditional on being unequal the orientations $01$ and $10$ are
equiprobable.  We adapt the classical von Neumann extractor
\cite{vonNeumann1951}.  The extracted bits from the sampling part of a block
are used as masks only at later target coordinates in that block.  Thus every
mask bit is known before the corresponding target source bit is revealed.

\begin{proposition}[Batched von Neumann extraction]
\label{prop:batched-vn}
For every computable $p\in(0,1)$ there are a total computable predictor $P$
and a computable schedule $S$ of positive asymptotic density such that
\begin{equation}
J_{P,S}^X\in\MLR
\end{equation}
for every $X\in\MLR_{\mu_p}$.
\end{proposition}

\begin{proof}
Since $q>0$ is computable, choose rationals $r,a$ with
\begin{equation}
0<2r<a<q.
\end{equation}
Hoeffding's inequality gives, for $c=2(a-r)^2>0$,
\begin{equation}
\mu_p\bigl(\operatorname{Bin}(N,q)<rN\bigr)\leq e^{-cN}
\end{equation}
for every $N$.  Fix an integer $M$ with $rM\geq1$ and put
\begin{equation}
N_k=M(k+1)^2,
\qquad
m_k=\lfloor rN_k\rfloor.
\end{equation}
Then $N_k\to\infty$, the series $\sum_k e^{-cN_k}$ converges computably, and
\begin{equation}
\frac{N_k}{\sum_{j<k}N_j}\longrightarrow0.
\end{equation}

Partition $\omega$ computably into successive blocks.  Block $k$ consists
first of a sampling segment of length $2N_k$, viewed as $N_k$ consecutive
pairs, and then a target segment of length $m_k$.  Let $S$ be the union of
the target segments.

On the sampling segment, apply the von Neumann rule
\begin{equation}
01\mapsto0,
\qquad
10\mapsto1,
\end{equation}
discarding $00$ and $11$.  If at least $m_k$ unequal pairs occur, let $U_k$
be the first $m_k$ extracted orientations.  Otherwise put
$U_k=0^{m_k}$.  At the $j$th target coordinate of block $k$, the predictor outputs
$U_k(j)$; at all other lengths, define it arbitrarily.  Since the entire
sampling segment precedes the target segment, this defines a total computable
predictor and is sequentially admissible.

Let $F_k$ be the event that fewer than $m_k$ unequal pairs occur.  The number
of unequal pairs has binomial distribution with parameters $N_k$ and $q$, so
\begin{equation}
\mu_p(F_k)\leq e^{-cN_k}.
\end{equation}

To verify randomness of the trace, let $A$ be an auxiliary fair-coin tape,
parsed into consecutive blocks $A_k\in2^{m_k}$.  Define the ideal mask
\begin{equation}
\widetilde U_k=
\begin{cases}
U_k,&\text{if }F_k\text{ does not occur},\\
A_k,&\text{if }F_k\text{ occurs}.
\end{cases}
\end{equation}
Conditional on success, the first $m_k$ von Neumann orientations are exactly
uniform and are independent of the target segment.  On failure, the auxiliary
block $A_k$ is uniform and independent of the source.  Hence each ideal mask
$\widetilde U_k$ is uniformly distributed on $2^{m_k}$ and independent of the
corresponding target segment.  Different blocks are independent.  Therefore,
when each target segment is XORed with its ideal mask, the resulting ideal
trace has fair-coin distribution.

Fix $X\in\MLR_{\mu_p}$ and choose $A\in\MLR^X$.  By van Lambalgen's theorem,
$(X,A)$ is Martin--L\"of random for $\mu_p\times\lambda$.  The ideal trace is
a computable image of $(X,A)$ whose pushforward measure is $\lambda$, so
randomness conservation makes it a member of $\MLR$.  The events $(F_k)$ are
uniformly effective and have a computably convergent sum.  Effective
Borel--Cantelli therefore implies that only finitely many $F_k$ occur on
$X$.  The actual trace differs from the ideal trace in only finitely many
blocks, hence in only finitely many bits, and is therefore also in $\MLR$.

Finally, put
\begin{equation}
L_K=\sum_{k<K}(2N_k+m_k).
\end{equation}
Since $m_k=\lfloor rN_k\rfloor$ and $N_k\to\infty$,
\begin{equation}
\frac{|S\cap L_K|}{L_K}
=
\frac{\sum_{k<K}m_k}{\sum_{k<K}(2N_k+m_k)}
\longrightarrow
\frac{r}{2+r}.
\end{equation}
Since the length of block $K$ is $o(L_K)$, passing from a block endpoint to
an intermediate position changes the density by $o(1)$.  Thus $S$ has
positive asymptotic density $r/(2+r)$.
\end{proof}

\begin{corollary}[Bernoulli-random presentations contain a fair-random SIR universe]
\label{cor:bernoulli-fair-random-richness}
Let $p\in(0,1)$ be computable and let $X\in\MLR_{\mu_p}$.  Then there is a
fair-coin Martin--L\"of-random real $R$ such that
\begin{equation}
R\lesir X.
\end{equation}
Consequently, every computable countable partial order embeds into the SIR
degrees strictly below $X$, represented entirely by fair-coin
Martin--L\"of-random reals.
\end{corollary}

\begin{proof}
By Proposition~\ref{prop:batched-vn}, some trace
$R=J_{P,S}^X$ belongs to $\MLR$, and hence $R\lesir X$.  Apply
Theorem~\ref{thm:bernoulli-poset-universality} with parameter $1/2$.
\end{proof}

\begin{remark}[Native and extracted random suborders]
Theorem~\ref{thm:bernoulli-poset-universality}, applied directly to
$X\in\MLR_{\mu_p}$ with parameter $p$, gives a universal computable suborder
below $X$ represented by $\mu_p$-Martin--L\"of-random reals.  The preceding
corollary gives another universal computable suborder below the same
presentation, represented by fair-coin Martin--L\"of-random reals.  Thus,
when $p\neq1/2$, the disjoint classes $\MLR_{\mu_p}$ and $\MLR$ nevertheless
occur together in the same SIR lower cone.  This is the sense in which the
SIR-degree structures associated with the two measures are intertwined; it
does not assert that arbitrary degrees represented by the two classes are
comparable.

The density supplied by the construction decreases as the source becomes
more biased.  Indeed, the supply of von Neumann mask bits is governed by
\begin{equation}
q=2p(1-p)
 =\frac12-2\left(p-\frac12\right)^2,
\end{equation}
and the construction gives density $r/(2+r)$ for any rational
$0<r<q/2$.  Hence the available positive-density lower bound decreases as
$|p-1/2|$ increases.  This is a feature of the present construction, not an
optimality statement.
\end{remark}

\begin{remark}[Rate refinements]
The batched von Neumann construction is deliberately not rate-optimal.  Its
purpose here is to produce a positive-density fair-random trace and thereby
connect the native Bernoulli-random and fair-random parts of the SIR degree
structure.  The pair extractor can be replaced by finite-block type-class
extraction in the manner of Elias~\cite{Elias1972}, while retaining the same
one-block causal delay, to obtain higher extraction densities.  We do not
pursue sharp rate optimization here.
\end{remark}

\section{Conclusion}

Sequential-innovation reducibility formalizes a causal restriction on
truth-table computation: each output is obtained from a fresh source bit after
the predictor has committed using only the preceding history.  The resulting
reducibility strictly refines truth-table reducibility and induces an exhaustive
geometric dichotomy.  The upper cone above the computable degree is a
truth-table spine, canonically isomorphic to the truth-table degrees, while its
complement consists exactly of the reservoir-immune degrees and is downward
closed under SIR.

Within the reservoir-immune region, the Martin--L\"of-random degrees form a
proper downward-closed substructure with a markedly different geometry.
Same-Turing-degree descent is governed by finite shifts, strict Turing-degree
drops contain densely ordered random intervals, and random columns need not
admit joins.  The two bridge constructions clarify how the regions interact:
$X\mapsto X^*$ moves an RI presentation to the greatest SIR representative of
the same truth-table degree, while $\emptyset'\oplus R$ combines c.e. and random
suborders in a single lower cone.

The Bernoulli application shows that every computable
$\mu_p$-Martin--L\"of-random presentation has a positive-density
fair-coin-random SIR trace.  It therefore contains both native
$\mu_p$-random and extracted fair-random copies of every computable
countable partial order.  The batched von Neumann construction is not
rate-optimal; Elias-type block methods provide possible quantitative
refinements.

\section*{Acknowledgments}

The author used OpenAI's ChatGPT 5.6 as an interactive research and writing
assistant during the development of this work, including to test and refine
arguments, organize proofs, and improve the exposition. All mathematical
claims, proofs, references, and final formulations were independently reviewed
and verified by the author, who takes full responsibility for the content of
the paper.

\end{document}